\documentclass{amsart}
\pdfoutput=1

\usepackage{amssymb}
\usepackage{amsthm}
\usepackage{amsfonts}
\usepackage{amsmath}
\usepackage{pmboxdraw}
\usepackage{verbatim} 
\usepackage{graphicx}
\usepackage{color}

\usepackage{tikz}
\usetikzlibrary{decorations.pathreplacing,calc}
\usepackage[colorlinks=true, citecolor=blue, filecolor=black, linkcolor=black, urlcolor=black]{hyperref}
\usetikzlibrary{decorations.pathreplacing}
\usepackage{cite}
\usepackage[normalem]{ulem}
\usepackage{subcaption}
\usepackage{bbm}
\usepackage{bm}
\usepackage{mathtools}
\usepackage{todonotes}
\usepackage{kantlipsum}
\allowdisplaybreaks

\newcommand{\RN}[1]{%
	\textup{\uppercase\expandafter{\romannumeral#1}}%
}

\def\ve{\varepsilon}

\def\Re{ \mathrm{Re}}

\def\C{\mathbb{C}}

\def\R{\mathbb{R}}

\newcommand{\re}{\operatorname{Re}}
\newcommand{\im}{\operatorname{Im}}

\theoremstyle{plain}
\newtheorem{thm}{Theorem}[section]

\newtheorem{lem}[thm]{Lemma}

\newtheorem{prop}[thm]{Proposition}

\theoremstyle{remark}

\newtheorem{rem}{Remark}

\numberwithin{equation}{section}

\begin{document}

\title[Numerical ranges of non-normal random matrices]{Numerical ranges of non-normal random matrices: \\ elliptic Ginibre and non-Hermitian Wishart ensembles}
\author{Sung-Soo Byun}
\address{Department of Mathematical Sciences and Research Institute of Mathematics, Seoul National University, Seoul 151-747, Republic of Korea}
\email{sungsoobyun@snu.ac.kr}

\author{Joo Young Park}
\address{College of Medicine, Seoul National University, Seoul 151-747, Republic of Korea}
\email{jyp531@snu.ac.kr}

\begin{abstract}
The numerical range of a non-normal matrix plays a central role as a descriptor of non-normal effects beyond spectral information. 
We study a class of fundamental non-Hermitian random matrix ensembles that interpolate between the Hermitian and non-Hermitian regimes. 
Our analysis focuses on the elliptic Ginibre ensemble and its chiral counterpart, as well as on non-Hermitian Wishart matrices. 
For each of these models, we explicitly characterise the geometry of the numerical range in the large-system limit. 
In particular, we show that for the elliptic Ginibre ensemble and its chiral version, the limiting numerical range is an ellipse, whereas for the non-Hermitian Wishart ensemble it is described by a non-elliptic envelope. Furthermore, we determine the numerical range of products of $n$ independent elliptic Ginibre matrices, which recovers, in the cases $n=1$ and $n=2$, the results for the elliptic Ginibre ensemble and the non-Hermitian Wishart ensemble at maximal non-Hermiticity, respectively. 
\end{abstract}

\maketitle


\section{Introduction and main results}

In random matrix theory, normal matrices are largely governed by their spectral statistics, a fact that allows their behaviour to be understood primarily through eigenvalue distributions. By contrast, non-normal matrices exhibit a range of genuinely new phenomena such as substantial eigenvector overlaps \cite{CM98,BD21,ATTZ20,CR22,Fyo18,ABN24}, sensitivity to perturbations \cite{BNST17,FT21}, and the emergence of pseudospectrum \cite{BGNTW15,MKS24}. These effects highlight intrinsic limitations of purely spectral descriptions and naturally motivate the study of operator-valued objects that go beyond the spectrum.

Among these objects, the \emph{numerical range} (also known as the \emph{field of values}) \cite{GR97,HJ94} plays a central role in capturing non-normal effects beyond the spectrum. 
For a square matrix $A \in \C^{N \times N}$, it is defined as 
\begin{equation} \label{def of numerical range}
W(A):= \{ (Ay,y): \| y\|_2=1\}. 
\end{equation}
While the numerical range coincides with the eigenvalue spectrum for normal matrices, in the non-normal case it is in general a strictly larger convex set that contains the eigenvalue spectrum \cite{John76}. As such, it provides a robust descriptor of non-normal behaviour and often reveals geometric or stability properties that remain invisible at the level of the spectrum alone. Moreover, the numerical range serves as an effective tool for analysing convergence rates of iterative schemes for solving linear systems \cite{Eier93,Tad25}.

On the other hand, non-Hermitian random matrix theory has attracted considerable attention in recent years, motivated by a broad range of applications and structural phenomena that go beyond the Hermitian setting; see e.g. \cite{BF25} and references therein. A fundamental model in this area is the complex Ginibre matrix, an \(N \times N\) matrix with i.i.d.\ complex Gaussian entries of mean zero and variance \(1/N\). A classical result for this model is the circular law, which states that the empirical eigenvalue distribution converges to the uniform measure on the centred unit disc as $N \to \infty$. Despite the long history of the Ginibre ensemble, its numerical range was investigated only relatively recently. It was shown in the seminal work \cite{CGLZ14} that the numerical range converges to the centred disc of radius $\sqrt{2}$ as $N \to \infty$. More recently, the sharp convergence rates and fluctuations for the numerical radius of random matrix with general i.i.d. entries were established in \cite{BC25}, building on earlier results on small deviation tail estimates and correlation–decorrelation transitions for Wigner matrices \cite{BCEHK25,BCEHK25a}.

Beyond the Ginibre matrix whose eigenvalue spectra follow the circular law, there exist several variants of non-Hermitian random matrices, for instance those with correlated entries. In such models, the limiting eigenvalue distribution typically departs from radial symmetry and instead converges to a non-circular measure supported on a domain of nontrivial geometry, often referred to as the \emph{droplet}. While the geometric properties of these spectral limits have been extensively studied (see e.g. \cite{BF25a} and references therein), the corresponding numerical ranges have so far not been systematically investigated.

In this work, we pursue this direction by analysing several well-studied non-Hermitian random matrix ensembles, with particular emphasis on the elliptic Ginibre ensemble and its chiral counterpart, as well as a non-Hermitian Wishart ensemble. These models provide natural non-Hermitian extensions of the classical Gaussian and Laguerre unitary ensembles (GUE and LUE, respectively) \cite{Fo10}.

We begin by introducing the random matrix models and their basic statistical properties. In these models, a non-Hermiticity parameter $\tau \in [0,1]$ plays a central role. 

\begin{itemize}
    \item \textbf{Elliptic Ginibre matrix.} The elliptic Ginibre matrix is one of the most extensively studied models interpolating between Hermitian and non-Hermitian random matrices. It is constructed by a linear interpolation between the Hermitian and anti-Hermitian parts of the Ginibre matrix: 
\begin{equation} \label{def of eGinibre}
X^{ \rm e }:= \frac{\sqrt{1+\tau}}{2}(G+G^*)+\frac{\sqrt{1-\tau}}{2}(G-G^*),
\end{equation} 
where $G$ is the complex Ginibre matrix. In particular, $X^{ \rm e }$ reduces to the Ginibre matrix when $\tau=0$, and to the GUE when $\tau=1$.
It is well known, under the name of the \emph{elliptic law}, that the eigenvalues of the elliptic Ginibre matrix converge to the uniform distribution on an ellipse
\begin{equation} \label{def of S for eGinUE}
 S^{ \rm e } :=  \Big\{ (x,y) \in \R^2: \Big( \frac{x}{1+\tau}\Big)^2 + \Big( \frac{y}{1-\tau} \Big)^2 \le 1 \Big\}. 
\end{equation}
See Figure~\ref{Fig_elliptic and chiral} (A)--(D). 
We refer the reader to \cite[Section 2.3]{BF25} for further references on this model. 
\smallskip 
\item \textbf{Chiral elliptic Ginibre matrix.} The chiral version of the elliptic Ginibre matrix was introduced in the context of quantum chromodynamics with chemical potential \cite{Os04,Ste96,AB10}. In this model, one introduces an additional non-negative parameter $\nu$, and considers two $N\times (N+\nu)$ random matrices $P$ and $Q$ with independent complex Gaussian entries of mean zero and variance $1/(2N)$. These matrices serve as building blocks for the correlated matrices
\begin{equation} \label{def of X1 X2} 
X_1:=\sqrt{1+\tau} \, P+\sqrt{1-\tau} \, Q, \qquad X_2:=\sqrt{1+\tau} \, P-\sqrt{1-\tau} \, Q. 
\end{equation}
The chiral elliptic Ginibre matrix is then defined as a random Dirac matrix of size $(2N+\nu)\times(2N+\nu)$:
\begin{equation} \label{def of Dirac matrix}
X^{ \rm ce }:= \begin{bmatrix}
0 & X_1
\\
X_2^*&0 
\end{bmatrix}.  
\end{equation}
When $\tau=1$ this model reduces to the standard Hermitian chiral GUE \cite[Section 3.1]{Fo10}. 
In order to describe the macroscopic behaviour of this model, one makes use of the scaling 
\begin{equation} \label{def of scaling nu}
\lim_{N \to \infty} \frac{\nu}{N} = \alpha \in [0,\infty). 
\end{equation}
Then it was shown in \cite[Corollary 2]{ABK21} that the random eigenvalues tend to occupy the compact set enclosed by the quartic curve with equation: 
	\begin{equation}   \label{def of S for Dirac}
S^{ \rm ce }:=\Big\{	(x,y) \in \R^2: (x^2+y^2)^2+  \frac{16\tau^2}{(1-\tau^2)^2} x^2y^2-2\tau(2+\alpha)(x^2-y^2) \le ( 1+\alpha-\tau^2 ) (1-(1+\alpha)\tau^2 ) \Big\}. 
	\end{equation}
In particular, when $\alpha = 0$, this compact set reduces to the ellipse defined in \eqref{def of S for eGinUE}. A notable feature of $S^{ \rm ce }$ is that for $\alpha > 0$, its topology exhibits a phase transition: if $\tau < 1/\sqrt{1+\alpha}$, the droplet is connected, whereas if $\tau > 1/\sqrt{1+\alpha}$, it consists of two connected components. See Figure~\ref{Fig_elliptic and chiral} (E)--(H). 
\smallskip 
\item \textbf{Non-Hermitian Wishart matrix.} The non-Hermitian Wishart matrix, also known as the sample cross-covariance matrix model, was introduced as a framework for analysing time series based on covariance matrices constructed from time-lagged correlation matrices; see e.g. \cite{KS10,BBD23}.
It is defined as 
\begin{equation}\label{def of nWishart}
X^{ \rm w }:=X_1 X_2^*, 
\end{equation}
where $X_1$ and $X_2$ are given by \eqref{def of X1 X2}. 
When $\tau=0$, the model reduces to the product of two rectangular Ginibre matrices, whereas in the case $\tau=1$ it reduces to the LUE. 
It was shown in \cite[Theorem 1]{ABK21} that the eigenvalues tend to occupy the shifted ellipse  
\begin{equation}
   S^{ \rm w } : = \Big\{ (x,y) \in \R^2: \Big( \frac{x-\tau(2+\alpha)}{(1+\tau^2)\sqrt{1+\alpha}} \Big)^2 + \Big( \frac{y}{(1-\tau^2)\sqrt{1+\alpha}} \Big)^2 \le 1 \Big\}.
\end{equation} 
See Figure~\ref{Fig_NWishart}. We refer the reader to \cite{BN24} for further references on this model. 
\end{itemize}
 
In this work, we determine the limiting numerical range for each of the three models introduced above.  We begin by presenting our results for the elliptic Ginibre matrix and its chiral counterpart.
For $a,b >0$, we write  
\begin{equation} \label{def of ellipse gen}
E_{a,b}:= \Big \{ (x,y)\in \R^2: (x/a)^2+(y/b)^2 \le 1 \Big\} . 
\end{equation} 
We also denote by  
\begin{equation}
    d_H(X, Y):= \max\Big\{ {\sup_{x \in X}d(x, Y)}, {\sup_{y \in Y}d(X, y)} \Big\}
\end{equation}
the Hausdorff distance between two subsets \(X\) and \(Y\) of the complex plane. 

\begin{thm}[\textbf{Numerical range of elliptic and chiral elliptic Ginibre matrices}] \label{Thm_elliptic and chiral} Let $\tau \in [0,1]$. 
\begin{itemize}
    \item[(i)] \textbf{\textup{(Elliptic Ginibre matrix)}} Let $X^{ \rm e }$ be the elliptic Ginibre matrix of size $N$. Then we have 
   \begin{equation}
   \lim_{N \to \infty} d_H(W(X^{ \rm e }), E_{a,b}) = 0,  
\end{equation}
almost surely, where 
\begin{equation} \label{def of a b eGinUE}
a \equiv a(\tau):= \sqrt{2(1+\tau)}, \qquad b \equiv b(\tau):= \sqrt{2(1-\tau)}. 
\end{equation}
  \item[(ii)] \textbf{\textup{(Chiral elliptic Ginibre matrix)}} Let $X^{ \rm ce }$ be the chiral elliptic Ginibre matrix of size $2N+\nu$. Assume that $\nu$ scales proportionally to $N$ as specified in \eqref{def of scaling nu}. Then we have 
   \begin{equation}
   \lim_{N \to \infty} d_H(W(X^{ \rm ce }), E_{a,b}) = 0,  
\end{equation}
almost surely, where 
\begin{equation} \label{def of a b chiral}
a \equiv a(\tau,\alpha):= \frac{\sqrt{1+\tau}(\sqrt{1+\alpha} + 1)}{\sqrt{2}} , \qquad b\equiv b(\tau,\alpha):= \frac{\sqrt{1-\tau}(\sqrt{1+\alpha} + 1)}{\sqrt{2}} .
\end{equation}
\end{itemize}
\end{thm}

\begin{figure}[t]
	\begin{subfigure}{0.24\textwidth}
		\begin{center}	
			\includegraphics[width=\textwidth]{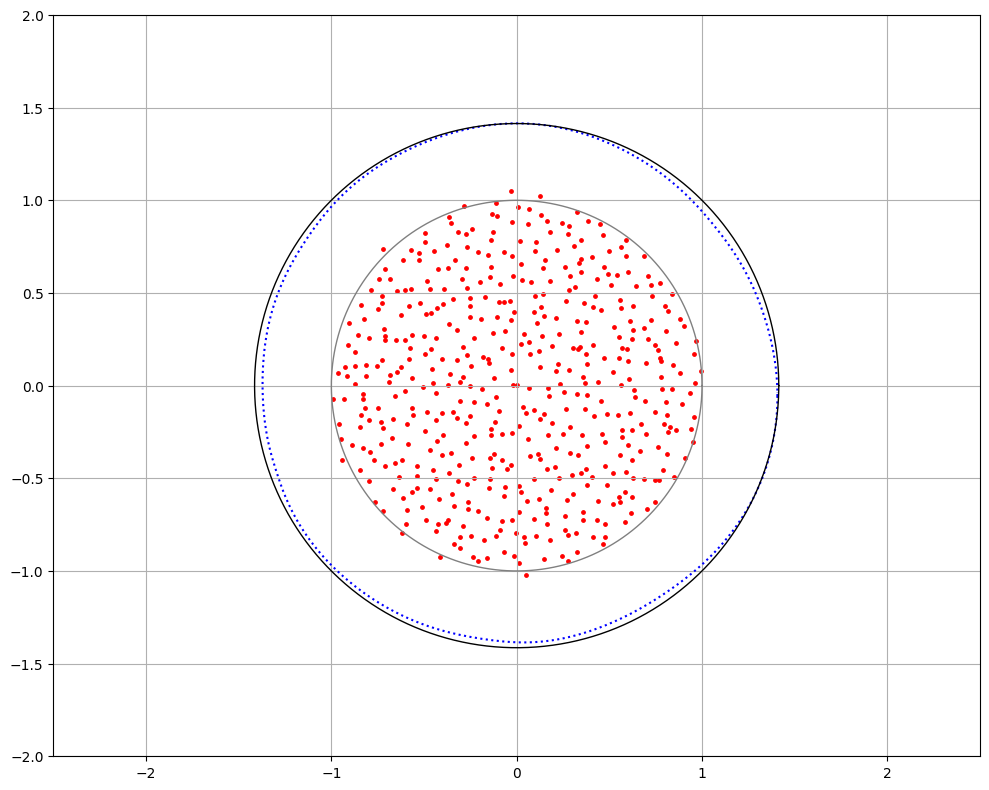}
		\end{center}
		\subcaption{$\tau=0$}
	\end{subfigure}	
		\begin{subfigure}{0.24\textwidth}
		\begin{center}	
			\includegraphics[width=\textwidth]{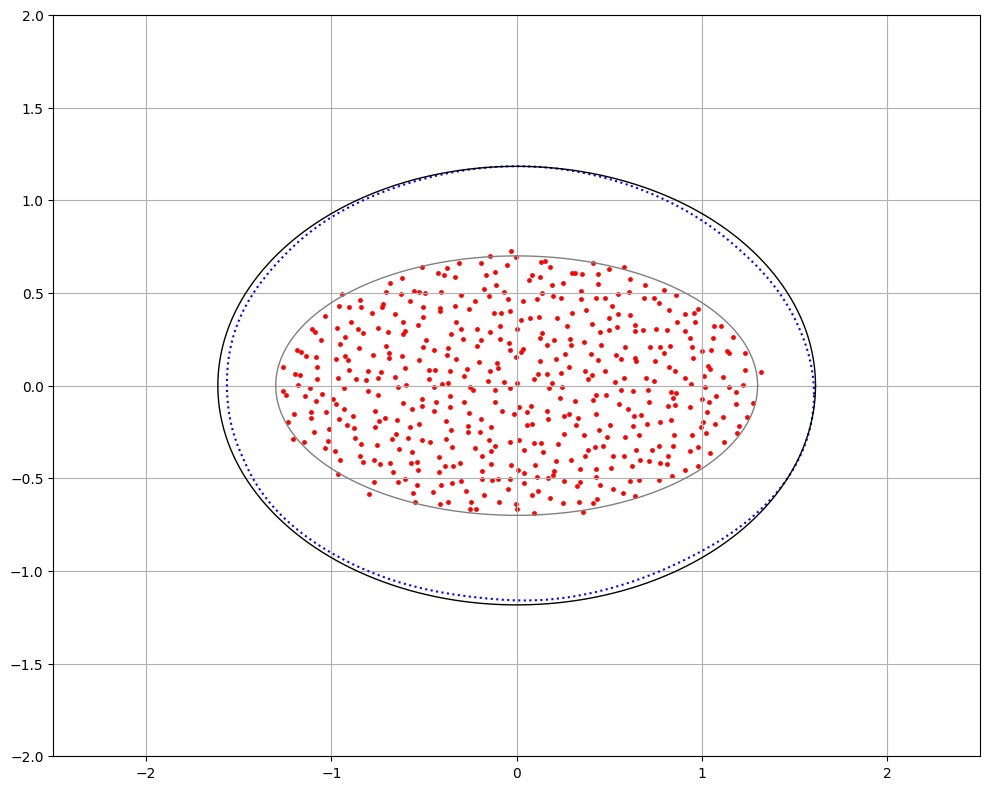}
		\end{center}
		\subcaption{$\tau=0.3$}
	\end{subfigure}
    	\begin{subfigure}{0.24\textwidth}
		\begin{center}	
			\includegraphics[width=\textwidth]{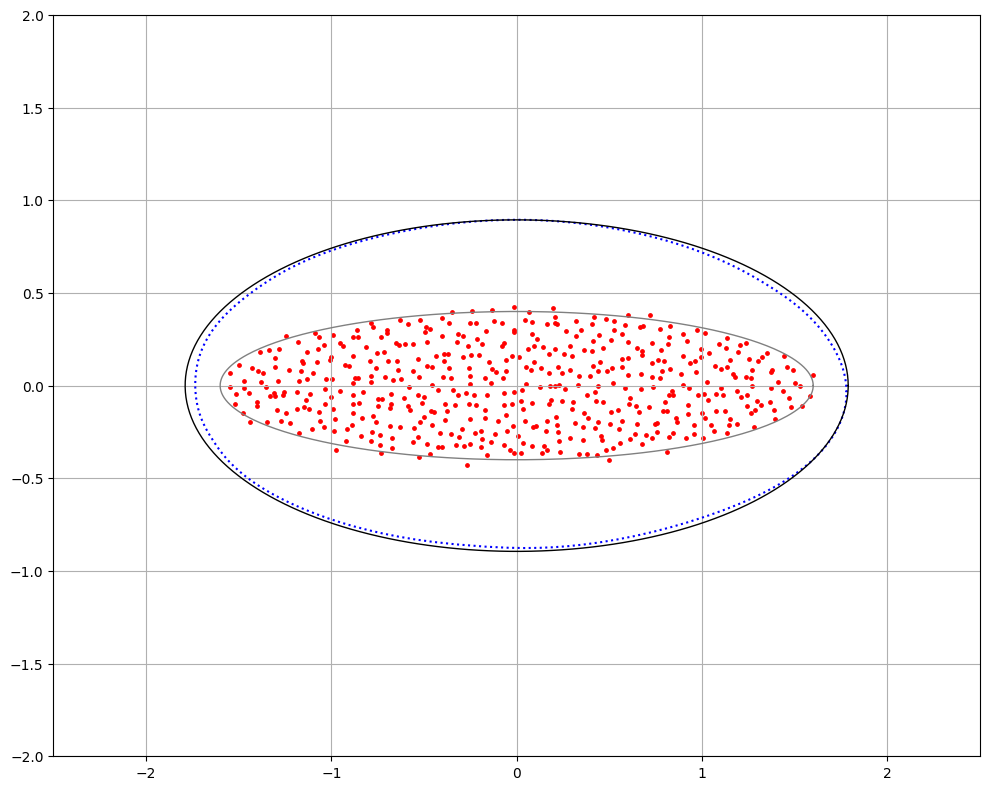}
		\end{center}
		\subcaption{$\tau=0.6$}
	\end{subfigure}
    	\begin{subfigure}{0.24\textwidth}
		\begin{center}	
			\includegraphics[width=\textwidth]{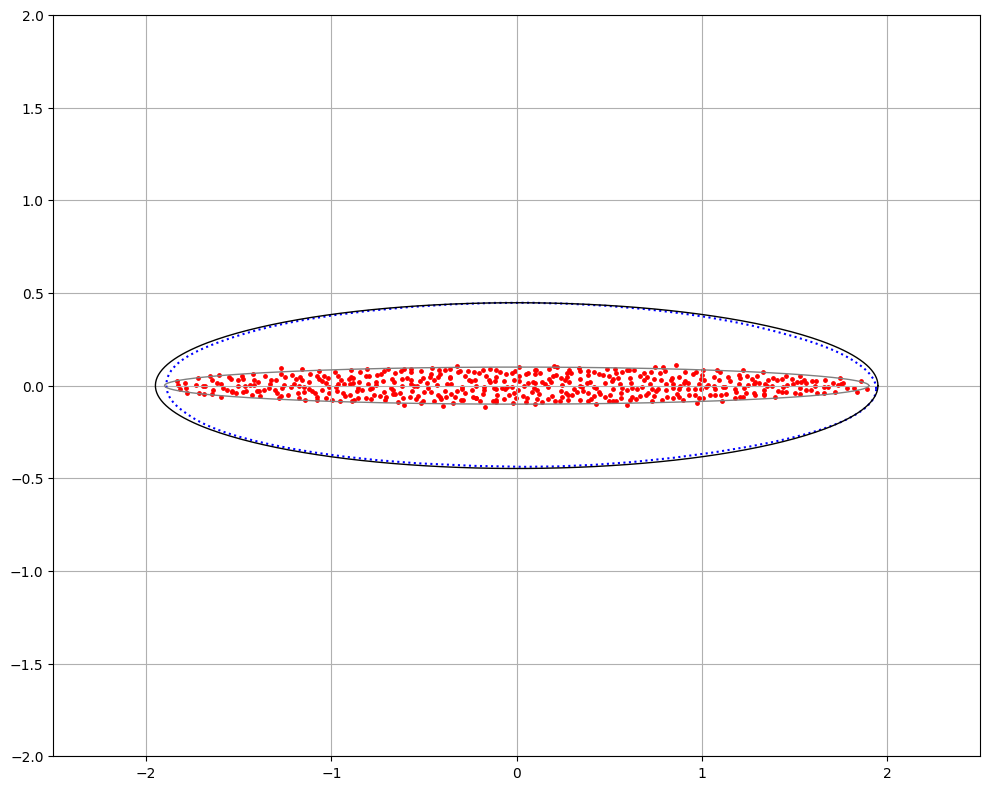}
		\end{center}
		\subcaption{$\tau=0.9$}
	\end{subfigure}

	\begin{subfigure}{0.24\textwidth}
		\begin{center}	
			\includegraphics[width=\textwidth]{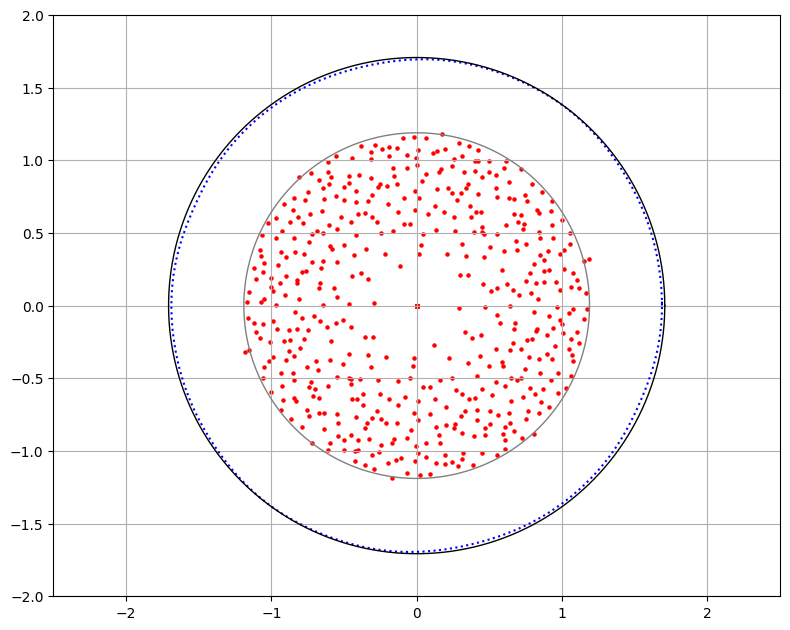}
		\end{center}
		\subcaption{$\alpha=1, \tau=0$}
	\end{subfigure}	
		\begin{subfigure}{0.24\textwidth}
		\begin{center}	
			\includegraphics[width=\textwidth]{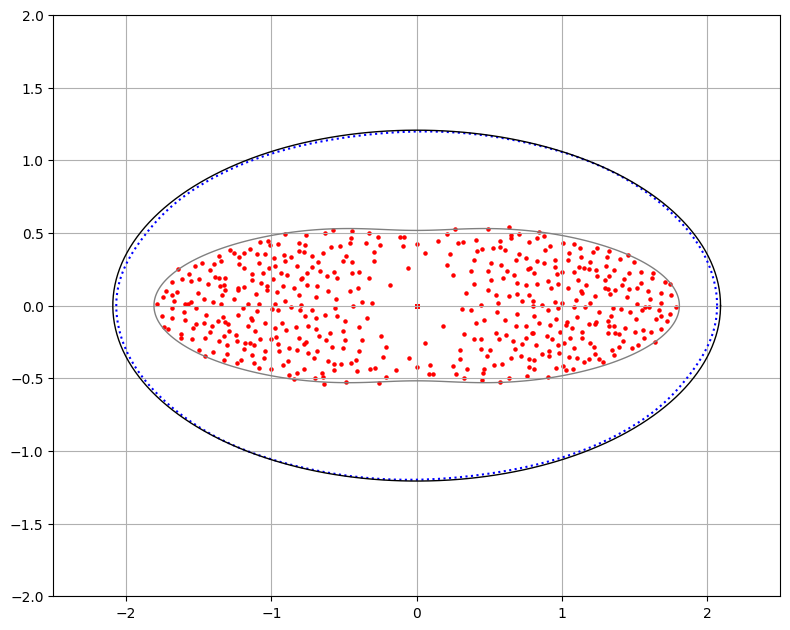}
		\end{center}
		\subcaption{$\alpha=1,\tau=0.5$}
	\end{subfigure}
    	\begin{subfigure}{0.24\textwidth}
		\begin{center}	
			\includegraphics[width=\textwidth]{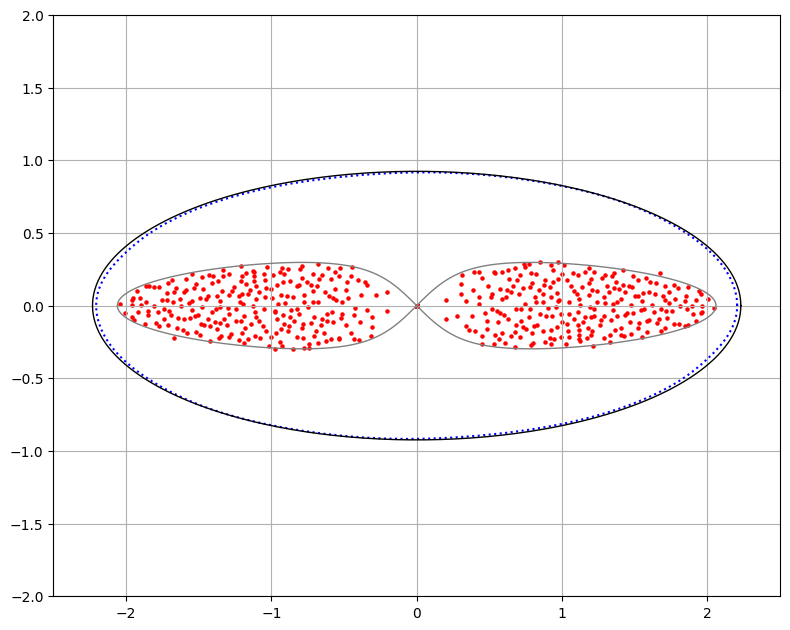}
		\end{center}
		\subcaption{$\alpha=1,\tau=1/\sqrt{2}$}
	\end{subfigure}
    	\begin{subfigure}{0.24\textwidth}
		\begin{center}	
			\includegraphics[width=\textwidth]{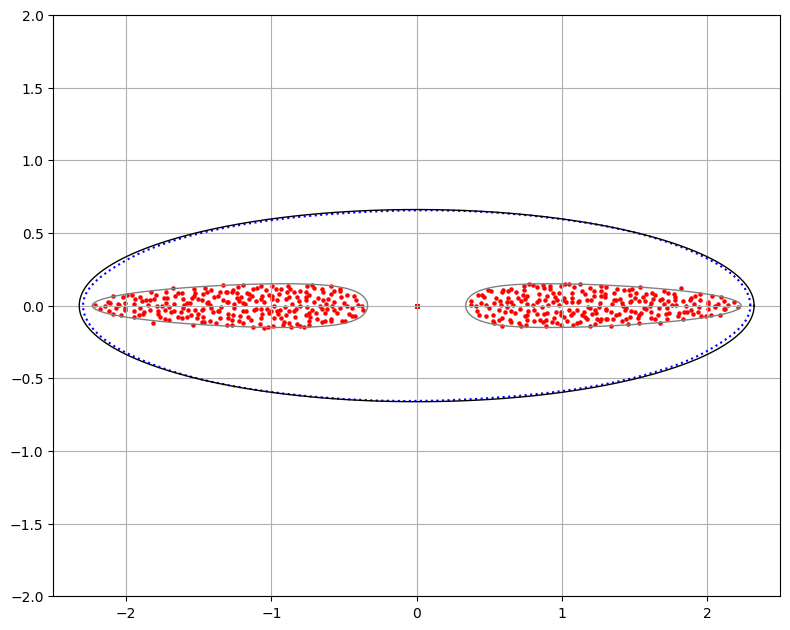}
		\end{center}
		\subcaption{$\alpha=1,\tau=0.85$}
	\end{subfigure}
	\caption{ The plots display the eigenvalues and numerical ranges of the elliptic Ginibre matrix ((A)--(D)) and the chiral elliptic Ginibre matrix ((E)--(H)). The red dots represent the eigenvalues together with the boundary of the droplet defined in \eqref{def of S for eGinUE} and \eqref{def of S for Dirac}, respectively. The solid black curves indicate the theoretical numerical ranges given in Theorem~\ref{Thm_elliptic and chiral}. The blue dotted curves show numerically computed numerical ranges, which are in good agreement with the theoretical results. Here, \(N=500\) for (A)--(D), while \(N=250\) for (E)--(H). }  \label{Fig_elliptic and chiral}
\end{figure}

See Figure~\ref{Fig_elliptic and chiral} for the numerics on Theorem~\ref{Thm_elliptic and chiral}. 
Note that, when $\alpha=0$, the values of $a$ and $b$ in \eqref{def of a b chiral} coincide with those in \eqref{def of a b eGinUE}. 

\begin{rem}[Outer and inner numerical radii]
In Theorem~\ref{Thm_elliptic and chiral}, the major and minor axes $a$ and $b$ are referred to as the outer and inner numerical radii, respectively. 
For the elliptic Ginibre matrix, the values of $a$ and $b$ in 
\eqref{def of a b eGinUE} were identified in \cite[Theorem~1.4]{BC25}, where the sharp error estimate and the corresponding fluctuation behaviour were also established. In contrast to the numerical radius, the statistics of extremal eigenvalues have been extensively studied. We refer to 
\cite{CESX22,CESX23,Be10,CJQ20,Ch22,HM25,XZ25} and the references therein 
for such results concerning various non-Hermitian random matrix ensembles.
\end{rem}

\begin{rem}[Extension beyond the Gaussian setting] \label{Rem_non Gaussian} In Theorem~\ref{Thm_elliptic and chiral}, the models are formulated using matrices with Gaussian entries. However, the argument extends without difficulty to ensembles with more general entries. The essential input in the proof is the almost sure convergence of the largest eigenvalue of the Hermitian part of the matrix.
In the case of the elliptic Ginibre ensemble, this Hermitian part belongs to the GUE, for which the convergence of the largest eigenvalue is classical. More generally, the same convergence holds for Wigner matrices under standard moment assumptions (see e.g. \cite{BSY88,BY93}). Consequently, Theorem~\ref{Thm_elliptic and chiral} remains valid beyond the Gaussian setting, aligning with the universal appearance of the elliptic law \cite{NO15,AK22}. 
\end{rem}

We now present our result for the non-Hermitian Wishart ensemble. Unlike the previous models, the resulting geometry is no longer simply described by an ellipse. To characterise it, we introduce a family of quartic polynomials parametrised by an angular variable. For $\theta \in [0,2\pi)$, set $\mathsf{c}:=\cos \theta$ and define 
\begin{equation}  \label{quartic polynomial of x among theta}
D_\theta(x):=a_4 x^4+a_3 x^3+a_2 x^2+a_1 x+a_0,
\end{equation}
where 
\begin{align}
\begin{split} 
a_4 & = 16\Big(1-\tau^2+\mathsf{c}^2\tau^2\Big),
\qquad a_3 =-32\,\mathsf{c}\tau(\alpha+2)\Big(1-\tau^2+\mathsf{c}^2\tau^2\Big),
\\
a_2 &= 16 \alpha^2\mathsf{c}^4\tau^4 + 4(\alpha^2-8\alpha-11) (1-\tau^2)^2 + 8\mathsf{c}^2\tau^2 ( 2 \alpha^2 -5 \alpha -6) (1-\tau^2),  
\\
a_1&= 4\,\mathsf{c}\tau(1-\tau^2) \Big( 2\alpha^2 \mathsf{c}^2 \tau^2 -(1-\tau^2)(2\alpha^3+5\alpha^2+8\alpha+3)  \Big) , 
\\
a_0 &=(2\alpha+1)^2(1-\tau^2)^2
\Big(\alpha^2\mathsf{c}^2\tau^2-(2\alpha+1)(1-\tau^2)\Big).
\end{split}
\end{align} 
The equation $D_\theta(x)=0$ has exactly two distinct real roots; see Lemma~\ref{Lem_real root count} below. 
We denote by $\lambda(\theta)$ the larger of these two roots.
We write $H_{\theta} := e^{-i \theta} \{ z \in \mathbb{C} : \re z \le \lambda(\theta) \}$, and define
\begin{equation} \label{def of tilde E}
    \widetilde{E} (\tau,\alpha) := \bigcap_{0 \le \theta \le 2 \pi} H_{\theta}.
\end{equation}
See Figure~\ref{Fig_non ellipse} for an illustration of this domain. 

\begin{thm}[\textbf{Numerical range of non-Hermitian Wishart matrix}] \label{Thm_NWishart} Let $\tau \in [0,1]$. Let $X^{ \rm w }$ be the non-Hermitian Wishart matrix of size $N$. Assume that $\nu$ scales proportionally to $N$ as specified in \eqref{def of scaling nu}. Then we have 
\begin{equation}
   \lim_{N \to \infty} d_H(W(X^{ \rm w }), \widetilde{E}(\tau,\alpha) ) = 0, 
\end{equation}
almost surely, where $\widetilde{E} (\tau,\alpha)$ is given by \eqref{def of tilde E}.
\end{thm}

See Figures~\ref{Fig_NWishart} and ~\ref{Fig_non ellipse} for the numerics on Theorem~\ref{Thm_NWishart}. 
The proof of Theorem~\ref{Thm_NWishart} proceeds via a structural reformulation of the Hermitian part $\re(e^{i\theta} X^{\mathrm w})$. Exploiting a spectral decomposition together with the Gaussian invariance of the ensemble, we reduce the problem to two Wishart-type components.
This approach contrasts with the elliptic Ginibre case in Remark~\ref{Rem_non Gaussian}, as it depends essentially on the Gaussian nature of the entries.

\begin{figure}[t]
	\begin{subfigure}{0.24\textwidth}
		\begin{center}	
			\includegraphics[width=\textwidth]{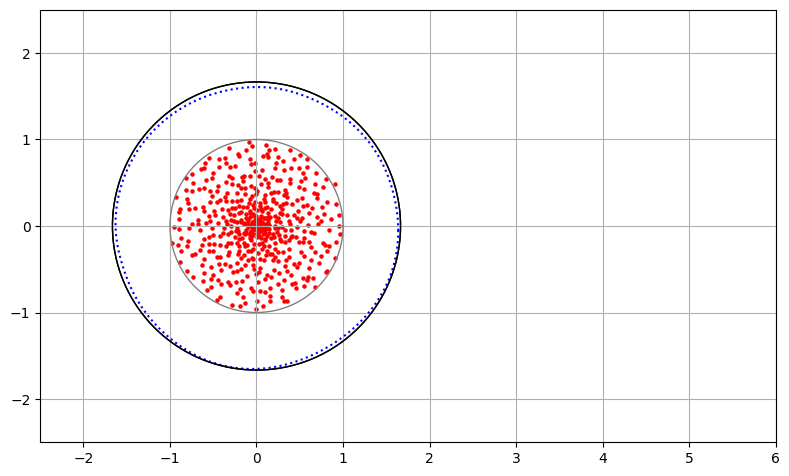}
		\end{center}
		\subcaption{$\alpha=0, \tau=0$}
	\end{subfigure}	
		\begin{subfigure}{0.24\textwidth}
		\begin{center}	
			\includegraphics[width=\textwidth]{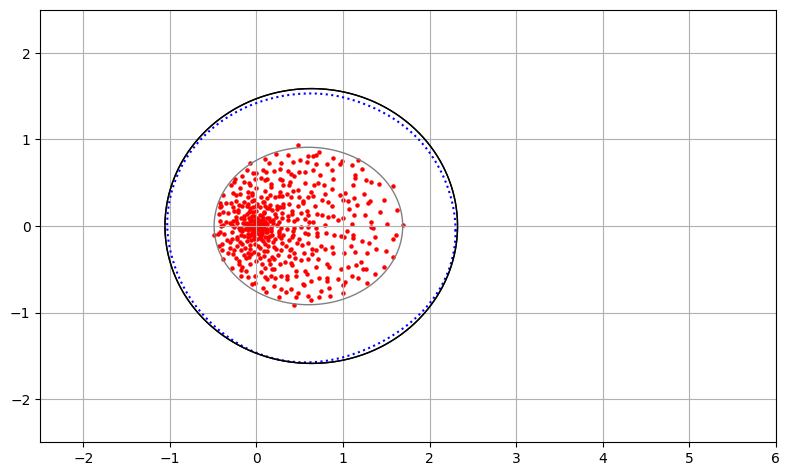}
		\end{center}
		\subcaption{$\alpha=0, \tau=0.3$}
	\end{subfigure}
    	\begin{subfigure}{0.24\textwidth}
		\begin{center}	
			\includegraphics[width=\textwidth]{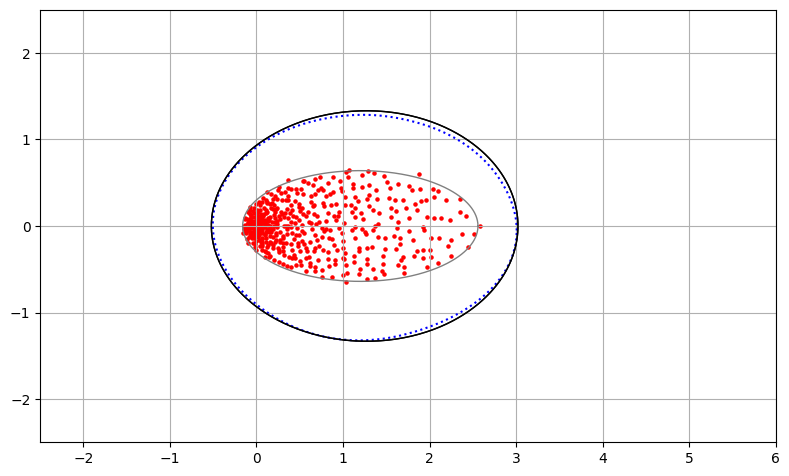}
		\end{center}
		\subcaption{$\alpha=0, \tau=0.6$}
	\end{subfigure}
    	\begin{subfigure}{0.24\textwidth}
		\begin{center}	
			\includegraphics[width=\textwidth]{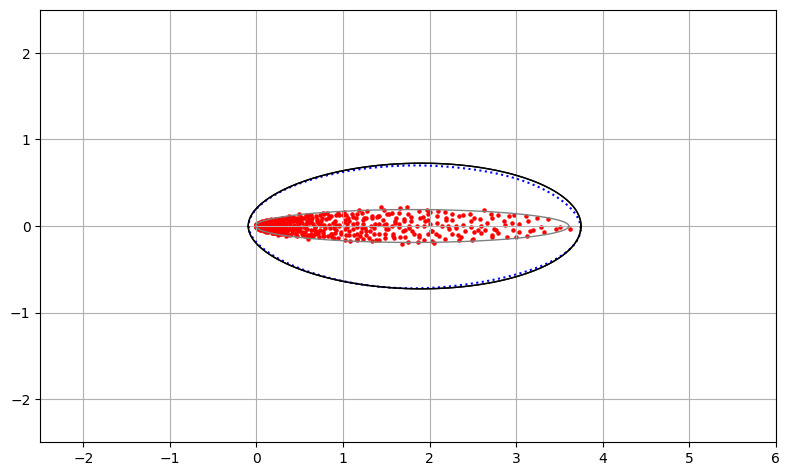}
		\end{center}
		\subcaption{$\alpha=0, \tau=0.9$}
	\end{subfigure}

	\begin{subfigure}{0.24\textwidth}
		\begin{center}	
			\includegraphics[width=\textwidth]{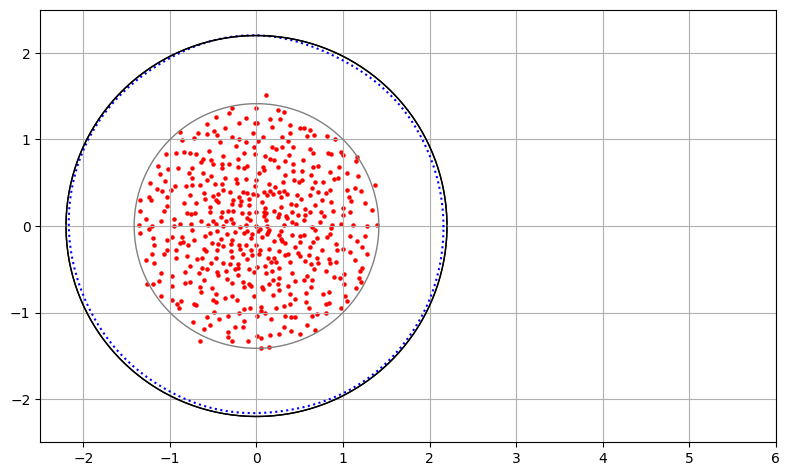}
		\end{center}
		\subcaption{$\alpha=1, \tau=0$}
	\end{subfigure}	
		\begin{subfigure}{0.24\textwidth}
		\begin{center}	
			\includegraphics[width=\textwidth]{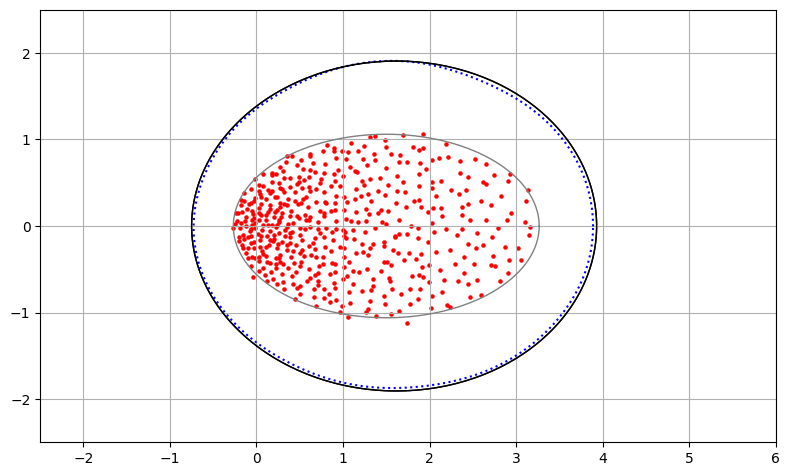}
		\end{center}
		\subcaption{$\alpha=1, \tau=0.5$}
	\end{subfigure}
    	\begin{subfigure}{0.24\textwidth}
		\begin{center}	
					\includegraphics[width=\textwidth]{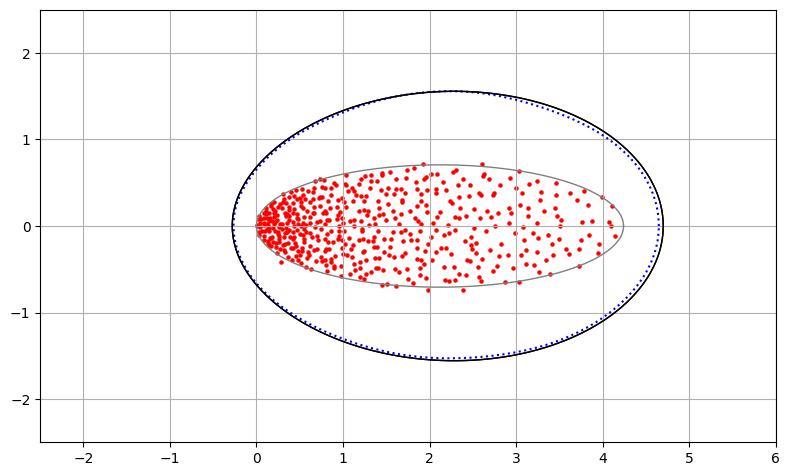}
		\end{center}
		\subcaption{$\alpha=1,\tau=1/\sqrt{2}$}
	\end{subfigure}
    	\begin{subfigure}{0.24\textwidth}
		\begin{center}	
					\includegraphics[width=\textwidth]{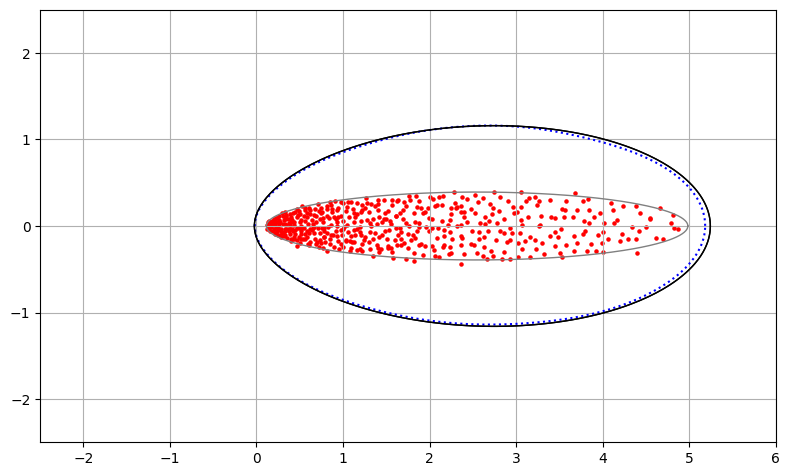}
		\end{center}
		\subcaption{$\alpha=1,\tau=0.85$}
	\end{subfigure}
	\caption{ The same figure as in Figure~\ref{Fig_elliptic and chiral} for the non-Hermitian Wishart ensemble. The solid black curve indicates the theoretical numerical range given in Theorem~\ref{Thm_NWishart}. Here, \(N=500\). }  \label{Fig_NWishart}
\end{figure}

\begin{rem}[Maximally non-Hermitian case; products of two rectangular Ginibre matrices] \label{Rem_products of rectangular}
In the special case $\tau=0$, the polynomial \eqref{quartic polynomial of x among theta} simplifies to
\begin{equation}
D_\theta(x)|_{\tau=0} = 16x^4+4(\alpha^2-8\alpha-11)x^2-(2\alpha+1)^3. 
\end{equation}
In this case, the expression is independent of the angular parameter $\theta$. 
The real roots of $D_\theta(x)|_{\tau=0}$ are given by $\pm B$, where
\begin{equation} \label{def of B}
B:= \sqrt{\frac{-\alpha^2+8 \alpha + 11 + (\alpha+5)\sqrt{\alpha^2+6 \alpha + 5}}{8}}. 
\end{equation} 
Consequently, it follows that  
\begin{equation}
\widetilde{E}(0,\alpha)= \mathbb{D}(B).   
\end{equation}
where  $\mathbb{D}(r)=\{z \in \C: |z| \le r \}$. 
\end{rem}

\begin{rem}[Hermitian limits]
In our models, taking the limit $\tau \to 1$ yields a Hermitian matrix ensemble. 
In this regime, the numerical range is expected to coincide with the convex hull of the eigenvalue spectrum, a phenomenon that can be directly verified from our  explicit results.

For the elliptic Ginibre ensemble, the spectral droplet defined in \eqref{def of S for eGinUE} collapses to the interval $[-2,2]$ as $\tau \to 1$. This interval is precisely the support of the semicircle law of the GUE. 
Consistently, from \eqref{def of a b eGinUE} we have $a(1)=2$, in agreement with this limiting behaviour.

For the chiral elliptic Ginibre ensemble, the spectral droplet defined in
\eqref{def of S for Dirac} converges, as $\tau \to 1$, to the union of two
disjoint intervals
\begin{equation} 
[-\sqrt{\lambda_+},\sqrt{\lambda_-}] \cup [ \sqrt{ \lambda_- }, \sqrt{ \lambda_+ } ], \qquad \lambda_\pm := (\sqrt{\alpha+1}\pm1)^2. 
\end{equation}
This set coincides with the support of the limiting spectral distribution of
the chiral GUE.
On the other hand, it follows from \eqref{def of a b chiral} that, in the same limit, the numerical range becomes the single interval $[-\sqrt{\lambda_+},\sqrt{\lambda_+} ]$, which is precisely the convex hull of the above two-cut spectral support.

For the non-Hermitian Wishart ensemble, as $\tau \to 1$, the spectral droplet
collapses to the interval $[\lambda_-,\lambda_+]$, which coincides with the
support of the Marchenko--Pastur law of the LUE.
To analyse the Hermitian limit of the numerical range, we observe that the
polynomial in \eqref{quartic polynomial of x among theta} simplifies to
$D(x)|_{\tau=1} = 16x^2(\alpha - x)^2 - 64x^3$. Its real roots are given by $x=0$ and $x=\lambda_\pm$. Consequently, the numerical range in the Hermitian limit is the interval $[\lambda_-,\lambda_+]$, in agreement with the spectral support.
\end{rem}

\begin{rem}[Geometry of numerical range of non-Hermitian Wishart matrix] \label{Rem_non ellipse}
In contrast to the elliptic and chiral elliptic Ginibre matrices considered in Theorem~\ref{Thm_elliptic and chiral} whose numerical ranges are given by ellipses, the numerical range of the non-Hermitian Wishart matrix in Theorem~\ref{Thm_NWishart} exhibits a markedly different behaviour.
Although it is a convex subset of the complex plane and bears a superficial resemblance to an ellipse, it is in fact not an ellipse; see Appendix~\ref{Appendix_non ellipse} for more discussions. 
\end{rem}

\begin{rem}[Numerical ranges of products and powers of complex Ginibre matrices] \label{Rem_products powers in main} As already noted in Remark~\ref{Rem_products of rectangular}, the non-Hermitian
Wishart matrix is closely related to products of Ginibre matrices; see
\cite[Section~2.7]{BF25}. In particular, when $\nu=0$, the non-Hermitian Wishart matrix coincides with the product of two independent square Ginibre matrices.
In this case, specialising \eqref{def of B} to $\alpha=0$, the numerical range is given by a centred disc of radius
\begin{equation} \label{def of numerical radius for 2 products}
\sqrt{ \frac{11+ 5\sqrt{5} }{ 8 }  } \asymp 1.665. 
\end{equation} 

It is instructive to compare this situation with the case of powers of a Ginibre matrix.
While the special case of the non-Hermitian Wishart matrix corresponds to
a product $Y_1 Y_2$ of two independent Ginibre matrices, it is fundamentally
different from a power of a single Ginibre matrix, such as $Y_1^2$.
Nevertheless, their limiting global eigenvalue distributions coincide \cite{AB12,LW16}.
Numerical simulations presented in Figure~\ref{Fig_products} (A)
suggest that their numerical ranges also coincide in the large-$N$ limit.
This phenomenon can indeed be established by adapting the same strategy
used in the proof of Theorem~\ref{Thm_NWishart}. 
We further observe that this agreement persists more generally:
the numerical ranges of products and powers of Ginibre matrices coincide
asymptotically, provided that the total number of Ginibre factors--counted with multiplicity--is the same; cf. Proposition~\ref{Prop_index invariance for products of GinUEs}. 
See Figure~\ref{Fig_products} (B)--(D) for the case of three Ginibre factors.

We also briefly comment on the value of the numerical radius in \eqref{def of numerical radius for 2 products}. In general, several inequalities are known for the numerical radius of a matrix. A classical upper bound, which follows from the submultiplicativity of the numerical radius, states that for any square matrix $X$, the numerical radius $r(X)$ satisfies 
\begin{equation} \label{upper bdd of powers of radius}
r(X^2) \le r(X)^2. 
\end{equation}
Since the value in \eqref{def of numerical radius for 2 products} is strictly smaller than $2$ (recalling that the numerical radius of the Ginibre matrix is $\sqrt{2}$ due to \cite{CGLZ14}), the bound \eqref{upper bdd of powers of radius} is not sharp for powers of Ginibre matrices.
\end{rem}

\begin{figure}[t]
	\begin{subfigure}{0.24\textwidth}
		\begin{center}	
			\includegraphics[width=\textwidth]{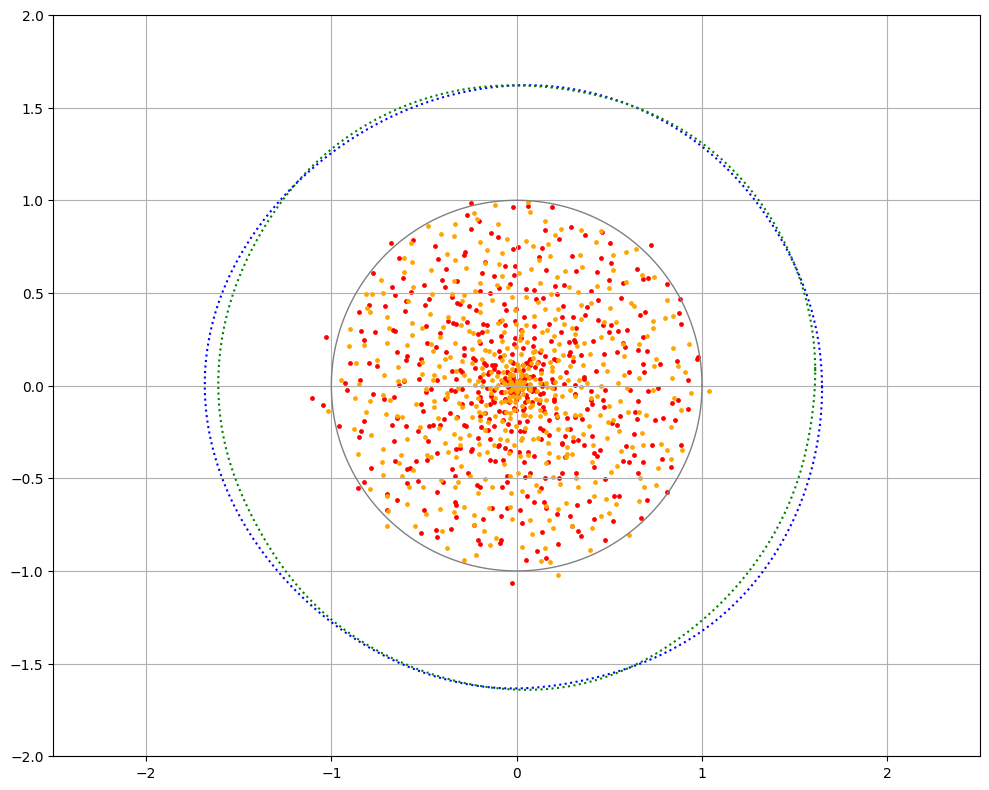}
		\end{center}
		\subcaption{$Y_1Y_2,\,Y_1^2$}
	\end{subfigure}	
		\begin{subfigure}{0.24\textwidth}
		\begin{center}	
			\includegraphics[width=\textwidth]{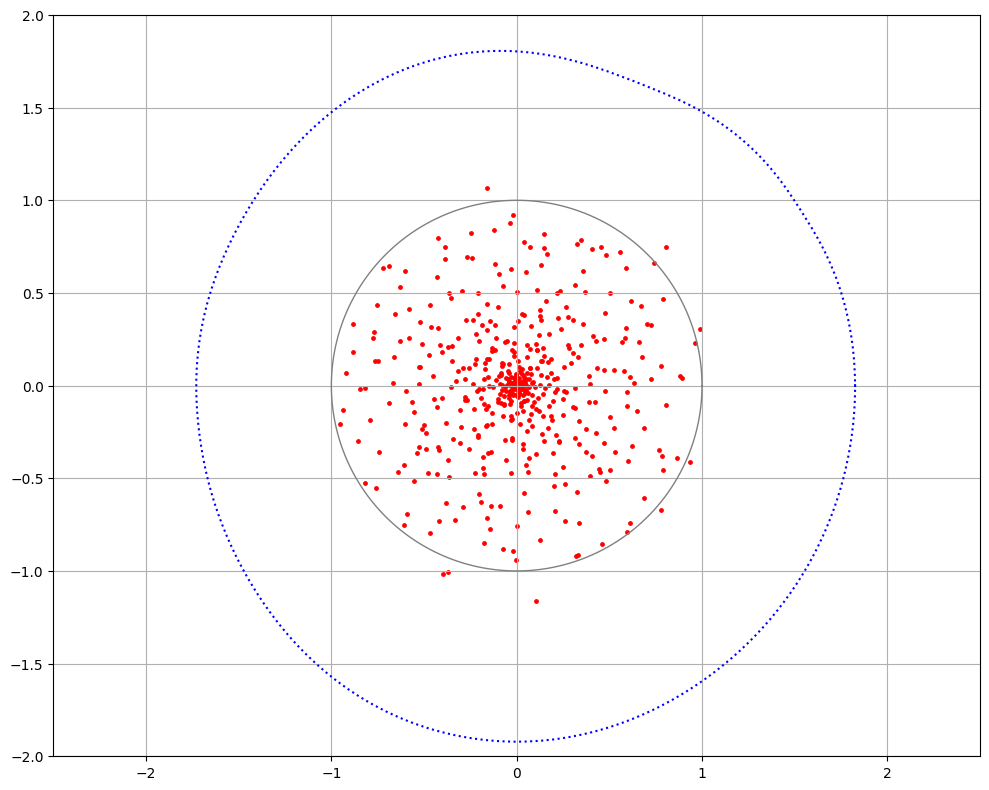}
		\end{center}
		\subcaption{$Y_1^3$}
	\end{subfigure}
    	\begin{subfigure}{0.24\textwidth}
		\begin{center}	
			\includegraphics[width=\textwidth]{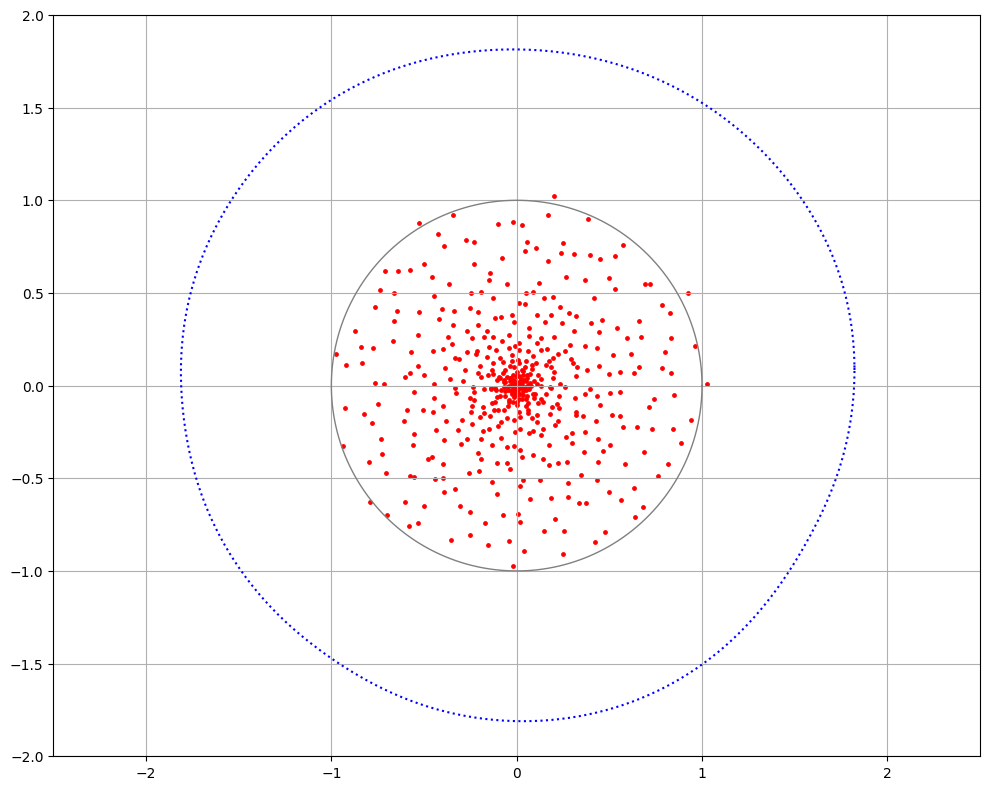}
		\end{center}
		\subcaption{$Y_1^2Y_2$}
	\end{subfigure}
    	\begin{subfigure}{0.24\textwidth}
		\begin{center}	
			\includegraphics[width=\textwidth]{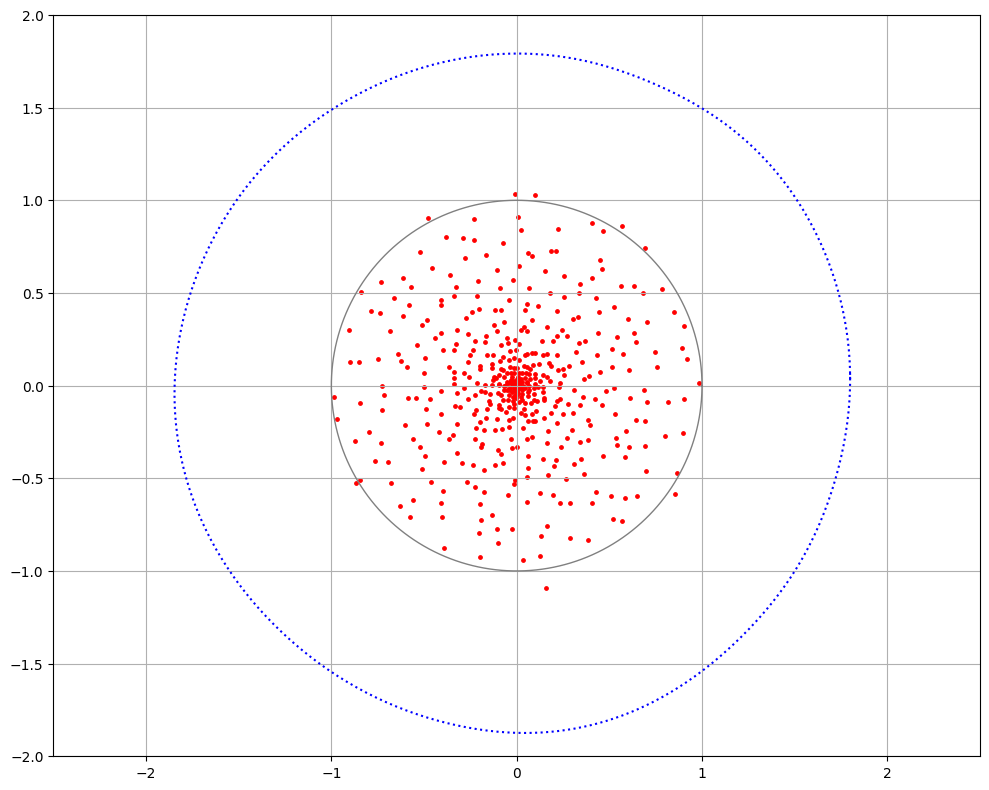}
		\end{center}
		\subcaption{$Y_1Y_2Y_3$}
	\end{subfigure}
	\caption{ The plots display the eigenvalues and simulated numerical ranges for products and powers of Ginibre matrices $Y_k$, where the matrices are normalised so that the associated droplet is the unit disc. In (A), we compare the product of two independent Ginibre matrices with the square of a single Ginibre matrix; in both cases, the numerical radius coincides with the value in \eqref{def of numerical radius for 2 products}. Figures~(B)--(D) show numerical ranges for various combinations of products and powers involving three Ginibre matrices. In all cases, the limiting numerical radius appears to be identical. Here, $N=500$.
 }  \label{Fig_products}
\end{figure}

Extending the previous remark, we now consider the product model of elliptic Ginibre matrices. 
Let $n \ge 2$ and let $X^{\rm e}_{1}, X^{\rm e}_{2}, \dots, X^{\rm e}_{n}$ be $N \times N$ i.i.d.\ elliptic Ginibre matrices. 
We define the product ensemble
\begin{equation} \label{def of product eGinUE}
\mathbf{X}_n^{ \rm e } := X_1^{ \rm e } X_2^{ \rm e } \dots X_n^{ \rm e }. 
\end{equation}
A remarkable fact shown in \cite{ORSV14} is that for $n \ge 2$, the limiting eigenvalue distribution of the product model \eqref{def of product eGinUE} does not depend on the non-Hermiticity parameter $\tau$. 
In particular, the limiting spectrum is supported on the unit disc, with a non-uniform density.

In the following result, we determine the numerical range of this product model. This demonstrates that the above $\tau$-independence phenomenon also persists at the level of the numerical range.

\begin{thm}[\textbf{Numerical range of products of eGinUEs}] \label{Thm_products eGinUE}
Let $\tau \in [0,1]$ and $n \ge 2$. Let $\mathbf{X}_n^{ \rm e }$ be given by \eqref{def of product eGinUE}. Then we have 
\begin{equation}
    \lim_{N \to \infty} d_H(W(\mathbf{X}^{\rm e}_n), \mathbb{D}(R_n) ) = 0, 
\end{equation}
almost surely, where  
\begin{equation} \label{def of Rn radius products}
 R_n:= \frac{ \sqrt{n} }{ 2^{ n+\frac32 }  } \bigg( 1+\sqrt{ 1+\frac{8}{n} } \, \bigg)^{ \frac32 }  \bigg( 3+\sqrt{ 1+\frac{8}{n} } \, \bigg)^{ \frac{n-1}{2} } .
\end{equation} 
\end{thm}

See Figure~\ref{Fig_elliptic_products} for a numerical verification of Theorem~\ref{Thm_products eGinUE}. 

\begin{figure}[t]
	\begin{subfigure}{0.24\textwidth}
		\begin{center}	
			\includegraphics[width=\textwidth]{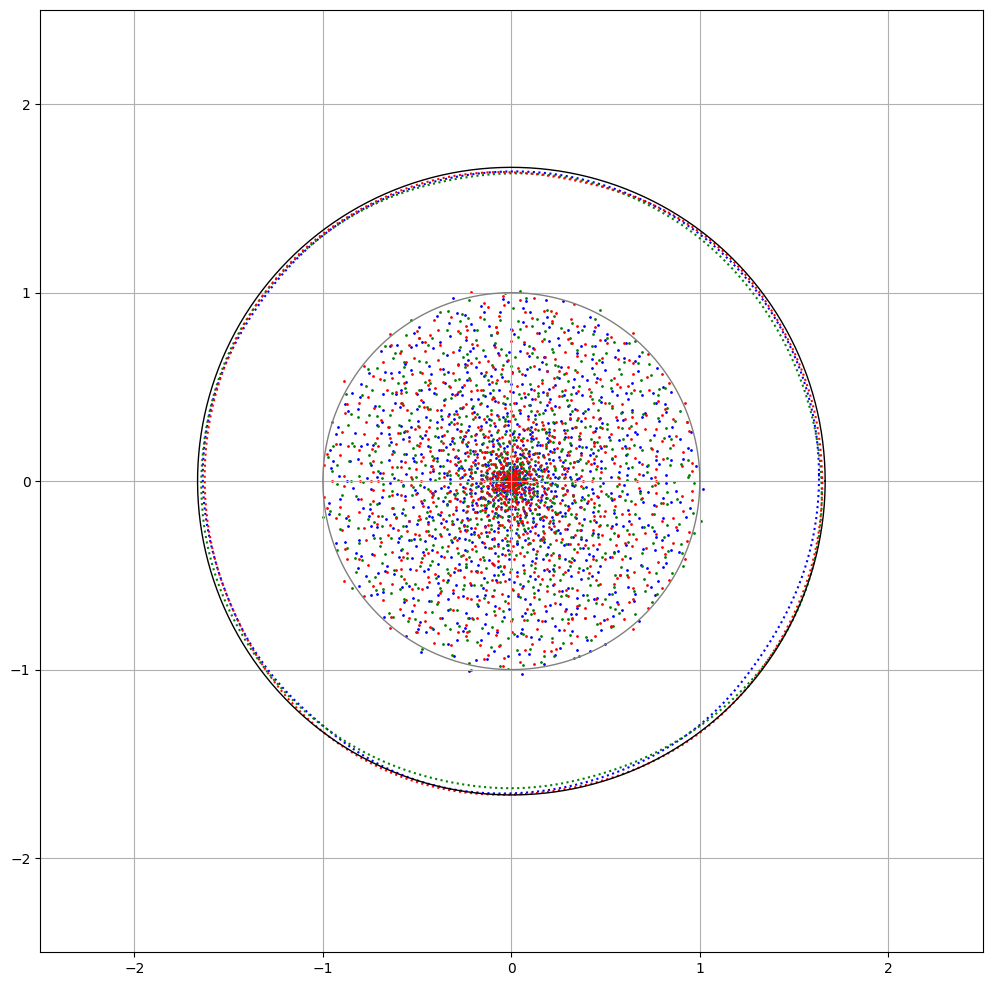}
		\end{center}
		\subcaption{$n=2$}
	\end{subfigure}	\qquad 
		\begin{subfigure}{0.24\textwidth}
		\begin{center}	
			\includegraphics[width=\textwidth]{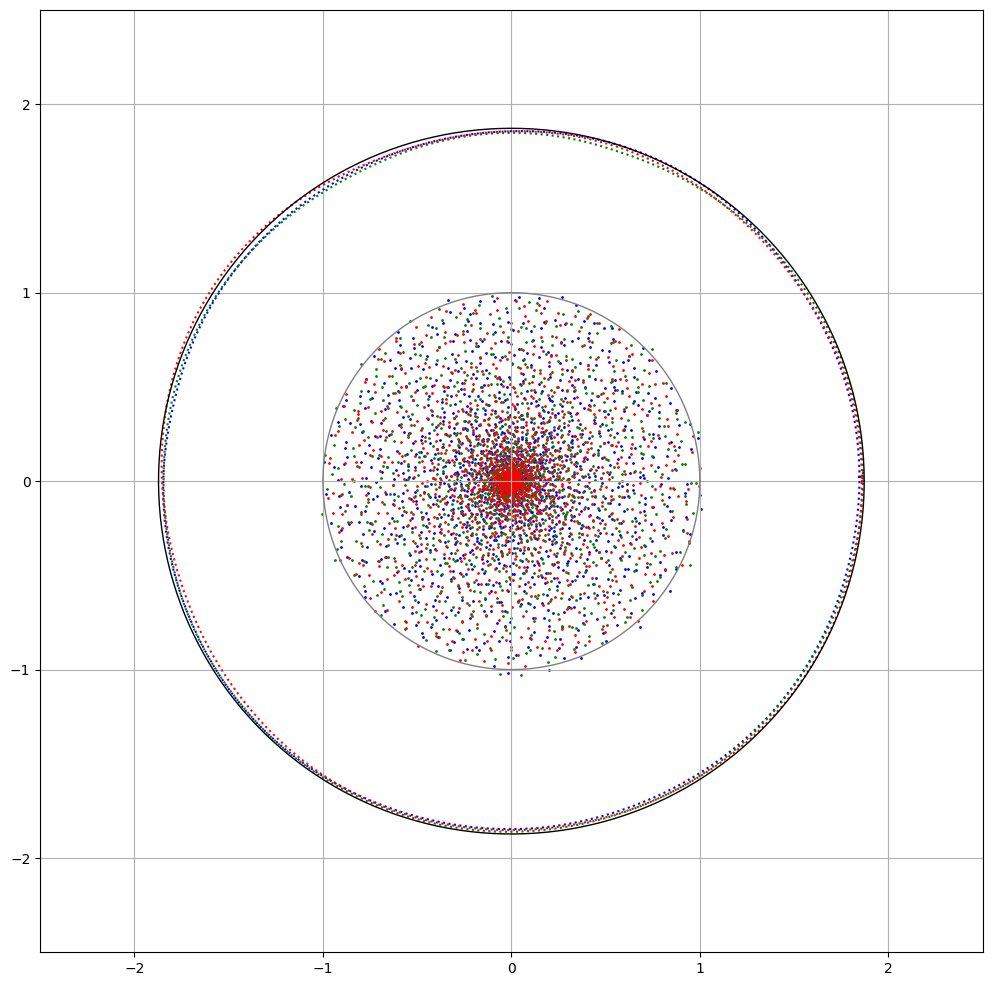}
		\end{center}
		\subcaption{$n=3$}
	\end{subfigure} \qquad 
    	\begin{subfigure}{0.24\textwidth}
		\begin{center}	
			\includegraphics[width=\textwidth]{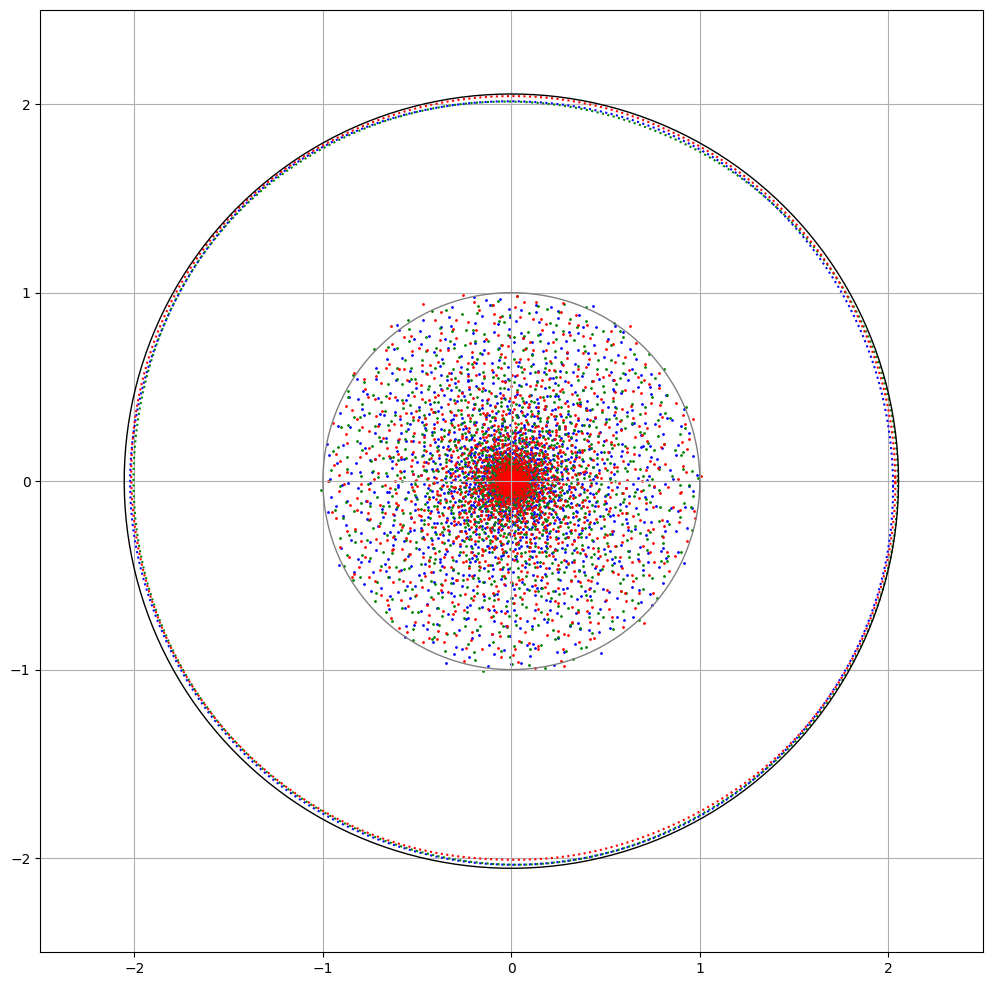}
		\end{center}
		\subcaption{$n=4$}
	\end{subfigure}
	\caption{ The plots display the eigenvalues (dots) and the numerical range (dotted curve) of $\mathbf{X}_n^{\mathrm e}$ for $n=2,3,4$ with $N=500n$, together with the circle of radius $R_n$ defined in \eqref{def of Rn radius products}. In each plot, we simultaneously present the cases $\tau = 0, 0.5, 1$, and observe that, for all these values of $\tau$, the numerical ranges asymptotically coincide with the analytic prediction. 
 }  \label{Fig_elliptic_products}
\end{figure}

Note that the first few values of $R_n$ are given by
\begin{equation}
R_2= \sqrt{ \frac{11+ 5\sqrt{5} }{ 8 }  } \approx 1.665,  \qquad R_3= \frac{ \sqrt{ 63+11\sqrt{33}  }  }{ 6 } \approx 1.872, \qquad R_4= \frac{ \sqrt{ 135+78\sqrt{3}  }  }{ 8 } \approx 2.054.  
\end{equation}
Notice that $R_2$ agrees with the value in \eqref{def of numerical radius for 2 products}. 
It is easy to see that $R_n$ is increasing in $n$, which is consistent with the intuition that the non-normality of products of random matrices increases as the number of factors grows. 
On the other hand, a straightforward computation shows that as $n \to \infty$, 
\begin{equation}
 R_n = \frac{ \sqrt{en}  }{ 2 } +O \Big( \frac{1}{ \sqrt{n} } \Big). 
\end{equation}

In Theorem~\ref{Thm_products eGinUE}, we focus on the product model \eqref{def of product eGinUE} formed from independent eGinUE matrices. 
However, in the GinUE case when $\tau=0$, one may also allow the matrices to appear with multiplicity, and the result continues to hold; see Proposition~\ref{Prop_index invariance for products of GinUEs}.

\subsection*{Organisation of the paper} The remainder of this paper is organised as follows.
In Section~\ref{Section_prelimiaries}, we collect the preliminary material and introduce key lemmas required for the proofs. 
Section~\ref{Section_proofs} is devoted to the proofs of Theorems~\ref{Thm_elliptic and chiral} and~\ref{Thm_NWishart}, while Section~\ref{Section_proof products} contains the proof of Theorem~\ref{Thm_products eGinUE}. 
Finally, in Appendix~\ref{Appendix_non ellipse}, we discuss Remark~\ref{Rem_non ellipse} in more detail with numerical simulations.

\subsection*{Acknowledgements}
Sung-Soo Byun was supported by the National Research Foundation of Korea grants (RS-2023-00301976, RS-2025-00516909). 
We thank Zhigang Bao and Giorgio Cipolloni for helpful comments during the preparation of the paper. 

\section{Preliminaries} \label{Section_prelimiaries}

In this section, we collect several basic properties of numerical ranges and present preliminary material needed for the proofs of our main results.

A remarkable property of the numerical range of a non-Hermitian matrix is that it admits an explicit characterisation in terms of its (rotated) Hermitian part. Specifically, let \(A\) be an \(N \times N\) matrix and \(\theta \in [0,2\pi)\), and denote by \(\lambda_{\mathrm{max}}(\theta,N)\) the largest eigenvalue of \(\re (e^{\mathrm{i}\theta}A)\). The numerical range \(W(A)\) can then be characterised as the intersection of a family of half-planes
\begin{equation} \label{numerical range in terms of half space inters}
W(A)=  \bigcap_{0 \le \theta \le 2\pi} H_{\theta, N}, \qquad H_{\theta, N} := e^{-i \theta} \{z \in \mathbb{C}: \re(z) \le \lambda_\text{max} (\theta, N) \} . 
\end{equation}
This property will play a key role in our analysis. Indeed, it was already used in a crucial way in \cite[Theorem 4.1]{CGLZ14}, where the following result was established. Let \(R>0\), and let \(\{X_N\}_{N\ge1}\) be a sequence of complex random \(N\times N\) matrices such that, for every \(\theta\in\mathbb{R}\), 
\begin{equation} \label{conv of re XN radially symmetric}
    \lim_{N\to\infty} ||\re (e^{i \theta}X_N)|| = R,
\end{equation}
almost surely, Then
\begin{equation} \label{conv to disc CGLZ14}
    \lim_{N\to\infty} d_H(W(X_N), \mathbb{D}(R)) = 0,
\end{equation}
almost surely. 
We next introduce a slight extension of this result to the situation in which
the limit~\eqref{conv of re XN radially symmetric} depends on the angular
parameter~$\theta$. This extension will be used in Theorem~\ref{Thm_elliptic and chiral}.

\begin{prop}[Elliptic numerical range] \label{Prop_elliptic numerical range}
Let \(a \ge b > 0\), and let \(\{X_N\}_{N\ge1}\) be a sequence of complex random \(N\times N\) matrices such that, for every \(\theta\in\mathbb{R}\), 
\begin{equation} \label{re rotated XN conv for ellipse}
    \lim_{N\to\infty} ||\operatorname{Re}(e^{i \theta}X_N)|| = \sqrt{a^2 \cos^2 \theta + b^2 \sin^2 \theta},
\end{equation}
almost surely. Then we have  
\begin{equation}
    \lim_{N\to\infty} d_H(W(X_N), E_{a,b}) = 0,
\end{equation} 
almost surely. Here, $E_{a,b}$ is an ellipse given by \eqref{def of ellipse gen}. 
\end{prop}

To establish such a lemma, it is convenient to formulate a general statement that allows one to apply the Toeplitz–Hausdorff theorem under uniform convergence. We present this statement below. 

\begin{lem} \label{Lem_uniform convergence}
Let \(\{X_N\}_{N\ge1}\) be a sequence of complex random \(N\times N\) matrices, and denote by \(\lambda_{\mathrm{max}}(\theta,N)\) the largest eigenvalue of \(\re (e^{{i}\theta}X_N)\). Suppose that, for every \(\theta\in[0,2\pi]\), 
\begin{equation}
    \lim_{N\to\infty} \lambda_{ \rm max } (\theta, N) = \lambda(\theta), 
\end{equation}
almost surely. For $\theta \in [0, 2\pi]$, define
\begin{align} 
 E := \bigcap_{0 \le \theta \le 2\pi} H_{\theta}, \qquad H_{\theta} := e^{-i \theta} \{z \in \mathbb{C}: \re(z) \le \lambda(\theta) \}. 
\end{align} 
Assume furthermore that \(\lambda(0)+\lambda(-\tfrac{\pi}{2})<\infty\). Then we have 
\begin{equation}
    \lim_{N\to\infty} d_H(W(X_N), E) = 0,
\end{equation}
almost surely. 
\end{lem}

\begin{proof} This statement provides a slight extension of \cite[Theorem~4.1]{CGLZ14}. The proof follows the same general strategy, with only minor modifications to the original argument. 
Note that 
\begin{align*}
    \limsup_{N \to \infty} ||X_N|| & \le \limsup_{N \to \infty} ||\re(X_N)|| + \limsup_{N \to \infty} ||\im(X_N)|| \\
    &= \limsup_{N \to \infty} ||\re(X_N)|| + \limsup_{N \to \infty} ||\re(e^{-i \frac{\pi}{2}}X_N)|| = \lambda(0) + \lambda(-\tfrac{\pi}{2}) < \infty. 
\end{align*}
Write $C:= \sup_{N\to\infty}||X_N|| < \infty$. 
Note that for all $N$ and $\theta, \phi \in [0, 2\pi]$, we have 
\begin{align*}
    |\lambda_{\operatorname{max}}(\theta, N) - \lambda_{\operatorname{max}}(\phi, N)| &\le ||\operatorname{Re}(e^{i \theta} X_N) - \operatorname{Re}(e^{i \phi} X_N)||  \le |e^{i \theta} - e^{i \phi}| \cdot ||X_N|| 
    \le  C|\theta-\phi|. 
\end{align*}
Therefore by taking the limit $N \to \infty$, we have $|\lambda(\theta) - \lambda(\phi)| \le C|\theta - \phi|$. 

Fix $\varepsilon>0$. Choose $\delta>0$ so that $C\delta<\frac{\varepsilon}{3}$, and let $\mathcal N=\{\theta_1,\dots,\theta_m\}\subset[0,2\pi]$ be a finite $\delta$–net of $[0,2\pi]$, i.e. for every $\theta\in[0,2\pi]$ there exists
$\theta_j\in\mathcal N$ such that $|\theta-\theta_j|\le\delta$.
By the pointwise convergence assumption, there exists $N_0\in\mathbb N$ such that for all $N\ge N_0$ and all $j\in\{1,\dots,m\}$, $|\lambda_{\max}(\theta_j,N)-\lambda(\theta_j)\big|<\frac{\varepsilon}{3}.$

Fix $N\ge N_0$ and an arbitrary $\theta\in[0,2\pi]$, and choose $\theta_j\in\mathcal N$
with $|\theta-\theta_j|\le\delta$. Using the Lipschitz continuity of
$\lambda_{\max}(\cdot,N)$ and $\lambda$, we obtain
\begin{align*}
    |\lambda_{\text{max}}(\theta, N) - \lambda(\theta)| &\le |\lambda_{\text{max}}(\theta, N) - \lambda_{\text{max}}(\theta_j, N)| + |\lambda_{\text{max}}(\theta_j, N) - \lambda(\theta_j)| + |\lambda(\theta_j) - \lambda(\theta)| 
    < \epsilon.
\end{align*}
Since $\theta$ was arbitrary, this proves the uniform convergence
\begin{equation}
\sup_{\theta\in[0,2\pi]}
\big|\lambda_{\max}(\theta,N)-\lambda(\theta)\big| \to 0,
\qquad \textup{as } N\to\infty .
\end{equation}

Define $h_N(\theta):=\lambda_{\max}(\theta,N)$. Then
$\|h_N-\lambda\|_\infty\to0$ as $N\to\infty$.
Let $z\in W(X_N)$. Note that for any $\theta\in[0,2\pi]$,
\[
\Re\big(e^{\mathrm i\theta}z\big)\le h_N(\theta)
\le \lambda(\theta)+\|h_N-\lambda\|_\infty
<\lambda(\theta)+\varepsilon
\]
for all sufficiently large $N$. Then by \eqref{numerical range in terms of half space inters}, $W(X_N)$ is contained in the $\ve$-neighourhood of $E$. Conversely, if $z\in E$, then for any $\theta\in[0,2\pi]$,
\[
\Re\big(e^{\mathrm i\theta}z\big)\le\lambda(\theta)
\le h_N(\theta)+\|h_N-\lambda\|_\infty
<h_N(\theta)+\varepsilon .
\]
Therefore $E$ is contained in the $\ve$-neighourhood of $W(X_N)$ for sufficiently large $N$.
This completes the proof.
\end{proof}

Building on Lemma~\ref{Lem_uniform convergence}, we show Proposition~\ref{Prop_elliptic numerical range}.

\begin{proof}[Proof of Proposition~\ref{Prop_elliptic numerical range}]
Note that by \eqref{re rotated XN conv for ellipse}, we have $\lambda(0) = a$ and $\lambda(-\frac{\pi}{2}) = b$.  
Let $l_\theta$ denote the line through the origin with argument $\theta\in[0,2\pi)$.
Among the two lines perpendicular to $l_\theta$ and tangent to the ellipse $E_{a,b}$, let $\widetilde l_\theta$ be the one lying in the direction of $l_\theta$.
Denote by $P_\theta$ the intersection point of $l_\theta$ and $\widetilde l_\theta$.
A direct computation shows that the distance from the origin to $P_\theta$ is given by
\begin{equation} \label{def of support function of ellipse}
|P_\theta| =\sqrt{a^2\cos^2\theta+b^2\sin^2\theta},
\end{equation} 
which coincides with \eqref{re rotated XN conv for ellipse}.
Therefore, by Lemma~\ref{Lem_uniform convergence}, the desired result follows. 
\end{proof}

We end this section by recalling a few basic notions from free probability theory that will be used in the proof of Theorem~\ref{Thm_NWishart}; see e.g. \cite{HP00,MS17,NS06,BS10} for further background.

Let $(\mathcal A,\phi)$ be a non-commutative probability space, where $\mathcal A$ is a unital $\ast$-algebra and $\phi$ is a tracial state.
Two self-adjoint elements $a_1,a_2\in\mathcal A$ are said to be \emph{free} if all their mixed centered moments vanish.
If $a_i$ has distribution $\mu_i$ $(i=1,2)$, the distribution of $a_1+a_2$
depends only on $\mu_1$ and $\mu_2$ and is called the \emph{free additive convolution} of $\mu_1$ and $\mu_2$, denoted by $\mu_1\boxplus\mu_2$. 

Free additive convolution naturally arises as the large-$N$ limit of sums of
independent random matrices.
More precisely, if $A_N$ and $B_N$ are independent Hermitian random matrices whose empirical spectral measures converge almost surely to compactly supported probability measures $\mu$ and $\nu$, respectively, and if $U_N$ is an independent Haar unitary matrix, then $A_N$ and $U_N B_N U_N^\ast$ are asymptotically free almost surely.
Consequently, the empirical spectral measure of $A_N+U_N B_N U_N^\ast$ converges almost surely to $\mu\boxplus\nu$; see e.g. \cite[Proposition~4.3.9]{HP00}.

In order to characterise $\mu\boxplus\nu$, it is convenient to work with analytic transforms.
For a compactly supported probability measure $\mu$ on $\mathbb R$, its
\emph{Cauchy transform} is defined by
\begin{equation}
G_\mu(z)=\int_{\mathbb R}\frac{1}{z-x}\,d\mu(x),
\qquad z\in\mathbb C\setminus\mathbb R.
\end{equation}
The measure $\mu$ can be recovered from $G_\mu$ via the Stieltjes inversion formula. 
The \emph{$R$-transform} of $\mu$ can be defined implicitly through the functional relation
\begin{equation} \label{def of R-G-relation}
G_\mu\Big(R_\mu(z)+\frac{1}{z}\Big)=z .
\end{equation} 
The central advantage of the $R$-transform is that it linearises free additive convolution: if $\mu$ and $\nu$ are compactly supported probability measures on
$\mathbb R$, then
\begin{equation} \label{R-additivity}
R_{\mu\boxplus\nu}(z)=R_\mu(z)+R_\nu(z).
\end{equation}

\section{Proofs of Theorems~\ref{Thm_products eGinUE} and ~\ref{Thm_NWishart}} \label{Section_proofs}

In this section, we present the proofs of the main results. 

For the reader’s convenience, we first recall a well-known result on the almost sure convergence of the extremal eigenvalues of a Wigner matrix; see e.g. \cite[Theorem~5.2]{BS10}. Suppose that the diagonal elements of the Wigner matrix $W_n =  (x_{ij})$ are i.i.d. real random variables, the elements above the diagonal are i.i.d. complex random variables, and all these variables are independent. Then, the largest eigenvalue of $W_n$ tends to $c_1$ and the smallest eigenvalue tends to $c_2$ almost surely if and only if $\mathbb{E}(x_{11}^2) < \infty $, $ \mathbb{E}(x_{12}) = 0$, $ \mathbb{E}(|x_{12}|^2) = \sigma^2/n< \infty$, and  $c_1 = 2 \sigma, c_2 = -2 \sigma$. 
The corresponding result for the sample covariance matrix can be found in \cite[Theorem 5.10]{BS10}; see also \cite{BSY88,BY93}.

\subsection{Proof of Theorem~\ref{Thm_elliptic and chiral}}

We first show the result for the elliptic Ginibre ensemble. 
By the definition of $X^{ \rm e }$ in \eqref{def of eGinibre}, we have 
\begin{equation}
    \re (e^{i \theta} X^{ \rm e }) = \frac{\sqrt{1+\tau}\cos(\theta)+i\sqrt{1-\tau}\sin(\theta)}{2}G + \frac{\sqrt{1+\tau}\cos(\theta)-i\sqrt{1-\tau}\sin(\theta)}{2}G^*. 
\end{equation}
Observe that this corresponds to a suitably rescaled GUE matrix (or, more generally, a Wigner matrix). Let $x_{ij}$ and $g_{ij}$ denote the $(i,j)$-th entries of $\re (e^{i \theta} X^{ \rm e })$ and $G$, respectively.
A direct computation yields
\begin{align}
   x_{11} &= \sqrt{1+\tau} \cos(\theta) \re(g_{11}) - \sqrt{1-\tau}\sin(\theta)\im (g_{11}),
   \\
     x_{12} &= \frac{\sqrt{1+\tau}\cos(\theta)+i\sqrt{1-\tau}\sin(\theta)}{2}g_{12} +  \frac{\sqrt{1+\tau}\cos(\theta)-i\sqrt{1-\tau}\sin(\theta)}{2} \overline{g_{21}} .
\end{align}
Consequently, we have 
\begin{align}
  \mathbb{E} \big[ x_{11}^2 \big] & = \frac{(1+\tau)\cos^2(\theta) + (1-\tau)\sin^2(\theta)}{2N} ,
  \\
   \mathbb{E} \big[ |x_{12}|^2 \big] &= \frac{(1+\tau)\cos^2(\theta)+(1-\tau)\sin^2(\theta)}{2N}, 
\end{align}
and $\mathbb{E} \big[ x_{12} \big] = 0$. 
Applying the general convergence result for the extremal eigenvalues of Wigner matrices mentioned above, we obtain 
\begin{equation}
    \lim_{N \to \infty} \left\| \re ( e^{i \theta} X^{ \rm e } ) \right\| = \sqrt{2(1+\tau) \cos^2 \theta + 2(1 - \tau) \sin^2 \theta}. 
\end{equation}
Then Theorem~\ref{Thm_elliptic and chiral} (i) now follows from Proposition~\ref{Prop_elliptic numerical range}.

We next prove the result for the chiral elliptic Ginibre ensemble. 
By the definition of $X^{\rm ce}$ in~\eqref{def of Dirac matrix}, we have
\begin{equation}
    \re(e^{i \theta} X^{ \rm ce }) = \begin{bmatrix} 0 & \sqrt{1+\tau} \cos(\theta) P + i \sqrt{1-\tau} \sin(\theta) Q \\ \sqrt{1+\tau} \cos(\theta) P^{*} - i \sqrt{1-\tau} \sin(\theta) Q^{*} & 0 \end{bmatrix}. 
\end{equation}
Since $P$ and $Q$ are independent rectangular Ginibre matrices with Gaussian entries, we have 
\begin{equation}
    \sqrt{1+\tau} \cos(\theta) P + i \sqrt{1-\tau} \sin(\theta) Q \stackrel{d}{=} \sqrt{(1+\tau) \cos^2(\theta) + (1-\tau) \sin^2(\theta)} P. 
\end{equation}
Consequently, it follows that 
\begin{equation}
    \re(e^{i \theta} X^{ \rm ce }) \stackrel{d}{=} \sqrt{(1+\tau) \cos^2(\theta) + (1-\tau) \sin^2(\theta)} \begin{bmatrix} 0 & P \\ P^{*} & 0 \end{bmatrix}. 
\end{equation}
Note here that the nonzero eigenvalues of the block matrix are given by $\pm \sqrt{\lambda_j}$, where $\{\lambda_j\}$ are the eigenvalues of $P P^{*}$.  Hence, its spectral norm is equal to $\|P\|$. By the almost sure convergence of the largest eigenvalue of a Wishart matrix to the right edge of the Marchenko--Pastur law, we obtain
\begin{equation}
  \lim_{N \to \infty} \left\| \re ( e^{i \theta} X^{ \rm ce } ) \right\| =  \frac{\sqrt{1+\alpha}+1}{\sqrt{2}}\sqrt{(1+\tau) \cos^2(\theta) + (1-\tau) \sin^2(\theta)}. 
\end{equation}
Then by Proposition~\ref{Prop_elliptic numerical range}, we obtain Theorem~\ref{Thm_elliptic and chiral} (ii). \qed

\begin{rem}[Alternative derivation for the elliptic Ginibre matrix]
While our proof provides a systematic framework for deriving the numerical range—particularly in the case of an elliptic domain—there is, for the elliptic Ginibre ensemble, a simpler argument available. 
Indeed, the numerical range may be recovered directly from the corresponding result for the Ginibre ensemble by exploiting the structural relation between the two models.

More precisely, by the definition \eqref{def of eGinibre}, for any $y \in \mathbb{C}^N$ with $\|y\|_2 = 1$, we have
\begin{align*}
    y^{*}X^{ \rm e }y 
    &= \frac{\sqrt{1+\tau}+\sqrt{1-\tau}}{2} z + \frac{\sqrt{1+\tau}-\sqrt{1-\tau}}{2}z^{*},
\end{align*}
where we have set $z := y^{*} G y$. 
Since $z$ belongs to the numerical range of $G$, the claim follows from the known description of this set.
More specifically, as $N \to \infty$, it follows from \cite{CGLZ14} that the numerical range of $G$ converges to the disc $\mathbb{D}(\sqrt{2})$. 
Hence we may write 
\[
z = r (\cos\theta + i \sin\theta),
\qquad
0 \le r \le \sqrt{2}, 
\quad \theta \in \mathbb{R}.
\]
Substituting into the above expression yields
\[
y^{*} X^{\mathrm e} y 
= r \sqrt{1+\tau} \cos\theta
  + i\, r \sqrt{1-\tau} \sin\theta,
\]
which parametrises precisely the ellipse stated in Theorem~\ref{Thm_elliptic and chiral} (i).
\end{rem}

\subsection{Proof of Theorem~\ref{Thm_NWishart}} In this subsection, we prove Theorem~\ref{Thm_NWishart}. 
Recall that the quartic polynomial $D_\theta$ is defined in~\eqref{quartic polynomial of x among theta}. 
In Theorem~\ref{Thm_NWishart}, the boundary of the numerical range is characterised in terms of the larger of the two real roots of $D_\theta$. 
We begin by showing that $D_\theta$ indeed has exactly two real roots. 

\begin{lem} \label{Lem_real root count}
For every $\alpha \ge 0$, $0<\tau<1$, and fixed $\theta$, the equation $D_\theta(x)=0$ has exactly two distinct real roots.
\end{lem}

\begin{proof}
Note that by \eqref{quartic polynomial of x among theta}, the leading coefficient of $D_\theta(x)$ is $16\bigl(1-\tau^2 \sin^2\theta\bigr) > 0,$ and a direct computation shows that $D_\theta\bigl(\tau \cos\theta \, \alpha\bigr) < 0.$
It follows that the equation $D_\theta(x)=0$ has at least two real roots.

Suppose, for contradiction, that there exist parameters $\alpha' \ge 0$ and $0<\tau'<1$ such that $D_\theta(x)=0$ has more than two real roots. 
By continuity of the roots with respect to the coefficients of the polynomial, there must then exist parameters $(\alpha,\tau)$ for which $D_\theta(x)=0$ has four real roots, with at least one of them being a double root. 
We now fix such a pair $(\alpha,\tau)$ and show that this leads to a contradiction.

For this, we first recall the definition of the resultant of two polynomials. 
Let
\[
\mathsf{P}(x)=a_m x^m+\cdots+a_0,
\qquad
\mathsf{Q}(x)=b_n x^n+\cdots+b_0
\]
be polynomials of degrees \(m\) and \(n\), respectively. The resultant of \(\mathsf{P}\) and \(\mathsf{Q}\) is defined as
\begin{equation}
\operatorname{Res}(\mathsf{P},\mathsf{Q}):=\det S(\mathsf{P},\mathsf{Q}),
\end{equation}
where the Sylvester matrix \(S(\mathsf{P},\mathsf{Q})\) is the \((m+n)\times(m+n)\) matrix defined by
\[
S(\mathsf{P}, \mathsf{Q} )=
\begin{pmatrix}
a_m & a_{m-1} & \cdots & a_0 & 0 & \cdots & 0 \\
0 & a_m & a_{m-1} & \cdots & a_0 & \ddots & \vdots \\
\vdots & \ddots & \ddots &  &  & \ddots & 0 \\
0 & \cdots & 0 & a_m & a_{m-1} & \cdots & a_0 \\
b_n & b_{n-1} & \cdots & b_0 & 0 & \cdots & 0 \\
0 & b_n & b_{n-1} & \cdots & b_0 & \ddots & \vdots \\
\vdots & \ddots & \ddots &  &  & \ddots & 0 \\
0 & \cdots & 0 & b_n & b_{n-1} & \cdots & b_0
\end{pmatrix}.
\]
Here, the first \(n\) rows consist of shifted copies of the coefficient vector \((a_m,\dots,a_0)\) and the last \(m\) rows consist of shifted copies of \((b_n,\dots,b_0)\). 
It satisfies the fundamental property that $\operatorname{Res}(\mathsf{P},\mathsf{Q})=0$ if and only if $\mathsf{P}$ and $\mathsf{Q}$ have a common root.  
In particular, a polynomial \(\mathsf{P}\) has a multiple root if and only if $\operatorname{Res}(\mathsf{P},\mathsf{P}')=0.$

Since $D_\theta(x)=0$ is assumed to have a real root of multiplicity two, 
the resultant of $D_\theta(x)$ and its derivative $D_\theta'(x)$ must vanish. 
A direct computation yields 
\begin{equation}
    \operatorname{Res}(D_{\theta}, D_{\theta}') = -2^{20}(\alpha+1)^2(1-\tau)^2(1+\tau)^2(1-\tau^2\sin^2(\theta))F(\alpha, \tau, \sin^2(\theta))^3
\end{equation}
where $F(\alpha,\tau,u)$ is an explicit polynomial obtained by straightforward computation:  
\begin{align}
\begin{split} 
    F(\alpha, \tau, u) & = -16 \alpha^3 \tau^6 u^3 + 24\alpha^2\tau^4\Big( \alpha \tau^2 + \alpha + \tau^2-1\Big)u^2 
    \\
    & \quad -3\tau^2
\Big( 4\alpha^3(\tau^2+1)^2 + \alpha^2(\tau^2-1)(17\tau^2-1) + (22\alpha+9)(\tau^2-1)^2 \Big) u 
    \\
    & \quad +2(\alpha+1)^3\,\tau^6 +3(\alpha+1)^2(2\alpha+7)\,\tau^4 +6(\alpha+1)(\alpha^2-11\alpha-8)\,\tau^2 +(2\alpha+1)(\alpha+5)^2. 
\end{split}
\end{align} 
Moreover, one can observe that $\partial F/\partial u$ is a quadratic polynomial in $u$, whose discriminant is given by
\begin{align}
\operatorname{Disc}\Big(\frac{\partial F}{\partial u}\Big) = -5184\alpha^3\tau^8(\alpha+1)^2(1-\tau)^2(1+\tau)^2 < 0 .
\end{align} 
This shows that $F(\alpha,\tau,u)$ is strictly decreasing function in $u$. 
Since $\sin^2\theta \in [0,1]$, it suffices to consider $0 \le u \le 1$. 
Evaluating at $u=1$, we obtain 
\begin{align}
    F(\alpha, \tau, 1) = (\alpha+5)^2(2\alpha+1)(1-\tau)^3(1+\tau)^3 > 0. 
\end{align}
By monotonicity, it follows that $F(\alpha,\tau,u)>0$ for all $0 \le u \le 1$. 
Consequently, $\operatorname{Res}(D_\theta,D_\theta') \neq 0,$ which contradicts the existence of a multiple real root.
\end{proof}

Before proceeding to the proof of Theorem~\ref{Thm_NWishart}, 
it is instructive to establish the following proposition, which provides a partial result by identifying the intersection of $W(X^{\rm w})$ with the real axis. This highlights the underlying structure of the two-matrix model and clarifies why free additive convolution plays a central role in the analysis. 
For this, we define
\begin{equation}  \label{quartic polynomial of x}
D(x): = D_\theta(x)|_{\theta=0},
\end{equation}
where $D_\theta$ is given by \eqref{quartic polynomial of x among theta}. 

\begin{prop} \label{Prop_Discriminant}
Let $ \alpha \ge 0 $ and $0 \le \tau < 1$. 
Then the intersection of $W(X^{\rm w})$ with the real axis converges almost surely to the real roots of the quartic equation $D(x)=0$.
\end{prop}

\begin{proof}  
By the definition of $X^{\rm w}$ in~\eqref{def of nWishart}, we have
\begin{equation}
\re (X^{ \rm w }) = (1+\tau)PP^{*} -(1-\tau)QQ^{*}. 
\end{equation} 
Note that both $(1+\tau)PP^{*}$ and $-(1-\tau)QQ^{*}$ are scaled Wishart matrices.

Let $\mu$ and $\nu$ denote the limiting eigenvalue distributions of 
$(1+\tau)PP^{*}$ and $-(1-\tau)QQ^{*}$, respectively. 
Each of these is a suitably rescaled Marchenko–Pastur distribution. 
Using the explicit form of the $R$-transform of the Marchenko–Pastur law 
(see e.g. \cite[Chapter~12]{NS06}), together with the scaling relation 
\begin{equation} \label{def of scaling relation of R}
R_{cm}(z) = cR_{m}(cz), 
\end{equation}
it follows that
\begin{align}
    R_{\mu}(z) = \frac{(1+\alpha)(1+\tau)}{2 - (1+\tau)z}, \qquad  R_{\nu}(z) = -\frac{(1+\alpha)(1-\tau)}{2 + (1-\tau)z}. 
\end{align}

Let $G$ denote the Cauchy transform of the free additive convolution 
$\mu \boxplus \nu$. 
By the additivity of the $R$-transform~\eqref{R-additivity} and the relation between the $R$-transform and the Cauchy transform~\eqref{def of R-G-relation}, we obtain
\begin{equation} \label{CYR}
    \frac{(1+\alpha)(1+\tau)}{2 - (1+\tau)G(z)} - \frac{(1+\alpha)(1-\tau)}{2 + (1-\tau)G(z)} + \frac{1}{G(z)} = z. 
\end{equation}
Rearranging \eqref{CYR} yields the cubic equation in $G(z)$:
\begin{equation} \label{Polynomial of Cauchy transform}
    z(1-\tau^2)G(z)^3 + \Big( (2\alpha + 1)(1-\tau^2)+4 \tau z \Big) G(z)^2 + (4 \tau \alpha - 4z)G(z) + 4 = 0.
\end{equation}
By the Stieltjes inversion formula, 
\begin{equation}
    \frac{d(\mu \boxplus \nu)}{dx}(x) = - \lim_{\epsilon \to 0} \im G(x+i \epsilon), 
\end{equation}
the endpoints of the support of $\mu \boxplus \nu$ correspond to those real values of $x$ for which the discriminant $D$ of the cubic equation~\eqref{Polynomial of Cauchy transform} vanishes. 
Indeed, when $D<0$, the equation has one real root and a pair of complex conjugate roots, whereas for $D \ge 0$ all roots are real.

Recall that the discriminant of the cubic polynomial is given by 
\begin{equation} \label{discriminant of cubic}
\operatorname{Disc}( ax^3+bx^2+cx+d )= b^2c^2 - 4ac^3 - 4b^3 d -27a^2d^2 + 18abcd. 
\end{equation}
Computing the discriminant of~\eqref{Polynomial of Cauchy transform} and dividing the resulting expression by $16$, we obtain a quartic polynomial in $z$, which coincides precisely with $D(z)$ defined in~\eqref{quartic polynomial of x}. Then, by the well-known almost sure convergence of the extremal eigenvalues for sums of independent Wishart matrices (see e.g. \cite{CD07,MS17} and \cite[Theorem~2.3]{FHS22}), together with an application of Lemma~\ref{Lem_uniform convergence} with $\theta=0$, the proof is complete. 
\end{proof}

We now prove Theorem~\ref{Thm_NWishart} by extending the previous proposition to a general angle $\theta$, making appropriate use of rotational invariance.

\begin{proof}[Proof of Theorem~\ref{Thm_NWishart}]
Let $M = N + v$ and fix $\theta \in [0,2\pi]$.  
Define $R := \begin{bmatrix} P & Q \end{bmatrix},$ which is an $N \times 2M$ complex random matrix, where $P$ and $Q$ are the rectangular Ginibre matrices used to define~\eqref{def of X1 X2}.
Next, set
\begin{equation}
T(\theta)
:=
\begin{bmatrix}
(1+\tau)\cos\theta & -i\sqrt{1-\tau^2}\sin\theta \\
i\sqrt{1-\tau^2}\sin\theta & -(1-\tau)\cos\theta
\end{bmatrix},
\qquad
S(\theta) := T(\theta) \otimes I_M.
\end{equation}
A direct computation shows that the eigenvalues of $T(\theta)$ are
\begin{equation} 
\lambda_{\pm}(\theta)
:= \tau\cos\theta \pm
\sqrt{1-\tau^2 \sin^2\theta}.
\end{equation} 
Consequently, $S(\theta)$ has the same two eigenvalues, each with multiplicity $M$.

Note that by \eqref{def of X1 X2} and \eqref{def of nWishart}, we have 
\begin{align*}
RS(\theta)R^{*} &= \begin{bmatrix} P & Q \end{bmatrix}
\begin{bmatrix} (1+\tau) \cos(\theta) I_M & -i \sqrt{1-\tau^2} \sin(\theta) I_M \\
    i \sqrt{1-\tau^2} \sin(\theta) I_M & -(1-\tau) \cos(\theta) I_M \end{bmatrix}
\begin{bmatrix}
    P^{*} \\ Q^{*}
\end{bmatrix} 
\\
&= \frac{1}{2} \begin{bmatrix} P & Q \end{bmatrix}
\begin{bmatrix}
    \sqrt{1+\tau} I_M & \sqrt{1+\tau} I_M \\
    \sqrt{1-\tau} I_M & -\sqrt{1-\tau} I_M
\end{bmatrix}
\begin{bmatrix}
    0 & e^{i \theta} I_M \\
    e^{-i \theta} I_M & 0
\end{bmatrix}
\begin{bmatrix}
    \sqrt{1+\tau} I_M & \sqrt{1-\tau} I_M \\
    \sqrt{1+\tau} I_M & -\sqrt{1-\tau} I_M
\end{bmatrix}
\begin{bmatrix}
    P^{*} \\ Q^{*}
\end{bmatrix} 
\\
& = \frac{1}{2} \begin{bmatrix} \sqrt{1+\tau}P + \sqrt{1-\tau}Q & \sqrt{1+\tau}P - \sqrt{1-\tau}Q \end{bmatrix}
\begin{bmatrix}
    0 & e^{i \theta} I_M \\
    e^{-i \theta} I_M & 0
\end{bmatrix}
\begin{bmatrix}
    \sqrt{1+\tau} P^{*} + \sqrt{1-\tau}Q^{*} \\ \sqrt{1+\tau}P^{*} - \sqrt{1-\tau}Q^{*}
\end{bmatrix}
\\
&= \frac{1}{2} \Big( e^{i \theta} X_1 X_2^{*} + e^{-i \theta} X_2 X_1^{*} \Big) = \re (e^{i \theta} X^{ \rm w }). 
\end{align*}
Since $S(\theta)$ is Hermitian, it admits the spectral decomposition 
\begin{equation}
S(\theta) = U(\theta) \begin{bmatrix} \lambda_{+}(\theta) I_M & 0 \\ 0 & \lambda_{-}(\theta) I_M  \end{bmatrix} U(\theta)^{*}, 
\end{equation}
where $U(\theta)$ is a deterministic $2M \times 2M$ unitary matrix. 
Consequently, 
\begin{align} \label{Diagonalize Eqn of Real part of X}
\re (e^{i \theta} X^{ \rm w })  = \widetilde{R} \begin{bmatrix} \lambda_{+}(\theta) I_M & 0 \\ 0 & \lambda_{-}(\theta) I_M  \end{bmatrix} \widetilde{R}^{*}, 
\end{align}
where $\widetilde{R} := R U(\theta)$. Viewing $R$ as consisting of $N$ independent rows of i.i.d.\ complex Gaussian vectors, the unitary invariance of the Gaussian distribution implies that $\widetilde{R}$ has the same distribution as $R$. 
Combining the above, we obtain 
\begin{align}
\re (e^{i \theta} X^{ \rm w }) & \stackrel{d}{=}\ R \begin{bmatrix} \lambda_{+}(\theta) I_M & 0 \\ 0 & \lambda_{-}(\theta) I_M  \end{bmatrix} R^{*} = \lambda_{+}(\theta)PP^{*} + \lambda_{-}(\theta)QQ^{*} .
\end{align}

The remainder of the argument parallels that of Proposition~\ref{Prop_Discriminant}. 
Define $\mu_\theta$ and $\nu_\theta$ to be the limiting spectral distributions of  $\lambda_{+}(\theta) P P^{*}$ and $\lambda_{-}(\theta) Q Q^{*}$, respectively. 
Using the $R$-transform of the Marchenko–Pastur law together with the scaling relation \eqref{def of scaling relation of R}, we obtain
\begin{align}
    R_{\mu_{\theta}}(z) = \frac{(1+\alpha)\lambda_{+}(\theta)}{2 - \lambda_{+}(\theta)z},  \qquad 
    R_{\nu_{\theta}}(z) = \frac{(1+\alpha)\lambda_{-}(\theta)}{2 - \lambda_{-}(\theta)z}. 
\end{align}
Let $G(z)$ denote the Cauchy transform of the free additive convolution 
$\mu_\theta \boxplus \nu_\theta$. 
Applying~\eqref{R-additivity} and~\eqref{def of R-G-relation}, we obtain
\begin{equation}
    z(1-\tau^2)G(z)^3 + \Big( (2\alpha + 1)(1-\tau^2)+4 \tau z \cos(\theta) \Big) G(z)^2 + (4 \tau \alpha \cos(\theta) - 4z)G(z) + 4 = 0. 
\end{equation}
Notice that for $\theta=0$, this reduces to \eqref{Polynomial of Cauchy transform}. 

Computing the discriminant of the above cubic polynomial with respect to $G(z)$ using \eqref{discriminant of cubic}, and dividing the resulting expression by $16$, we obtain--after lengthy but straightforward computations--the quartic equation $D_\theta(x)=0$, where $D_\theta$ is defined in~\eqref{quartic polynomial of x among theta}. 
Then again, by the almost sure convergence of the extremal eigenvalues for sums of independent Wishart matrices, it follows that the largest real root $\lambda(\theta)$ of $D_\theta$ coincides with the almost sure limit of $\lambda_{\max}(\theta,N)$. 
The theorem then follows from Lemma~\ref{Lem_uniform convergence}. 
\end{proof}

\section{Proof of Theorem~\ref{Thm_products eGinUE}} \label{Section_proof products}

In this section, we prove Theorem~\ref{Thm_products eGinUE}. 
The proof proceeds in two steps. First, in Lemma~\ref{tau invariance of the num range of product eGinibre}, we show that for $n \ge 2$, the numerical range of the product model $\mathbf{X}^{\rm e}_n$ defined in \eqref{def of product eGinUE} converges, in the large-$N$ limit, to a disc whose radius is independent of the non-Hermiticity parameter. 
We then determine the explicit form of the radius $R_n$, which yields the value given in \eqref{def of Rn radius products}.

\begin{lem} \label{tau invariance of the num range of product eGinibre}
For $n \ge 2$, let $\mathbf{X}^{\rm e}_n$ be defined as in \eqref{def of product eGinUE}. Then
\begin{equation}
    \lim_{N \to \infty} d_H(W(\mathbf{X}^{\rm e}_n), \mathbb{D}(R_n) ) = 0, 
\end{equation}
almost surely, where $R_n$ is a constant depending only on $n$. 
\end{lem}
\begin{proof}
Note that by \eqref{def of eGinibre}, the elliptic Ginibre matrix can be decomposed as 
\begin{equation} \label{redefine eGinibre}
    X^{ \rm e } = \sqrt{\frac{1+\tau}{2}}S+i \sqrt{\frac{1-\tau}{2}} T, 
\end{equation}
where $S$ and $T$ are independent GUE matrices.
For given semicircular elements $s_1, \dots, s_n, t_1, \dots, t_n$, we define 
\begin{equation} \label{def of uk and xn e}
    u_k := \sqrt{\frac{1+\tau}{2}} s_k + i \sqrt{\frac{1-\tau}{2}} t_k, \qquad  x^{\rm e}_n := \re \big(e^{i \theta} \prod_{k=1}^{n} u_k \big). 
\end{equation}
For $N \times N$ i.i.d.\ GUE matrices $X_1, X_2, \dots, X_n$ and free semicircular elements $s_1, s_2, \dots, s_n$, it was shown in \cite{HT05} that 
\begin{equation} \label{strong conv to semicircular}
    \lim_{N \to \infty} \| p(X_1, X_2, \dots, X_n) \| = \|p(s_1, s_2, \dots, s_n)\|
\end{equation}
almost surely, where $p$ is an arbitrary $*$-polynomial in $n$ variables. 
By \eqref{redefine eGinibre}, \eqref{def of uk and xn e} and \eqref{strong conv to semicircular}, we have 
\begin{equation}
    \lim_{N \to \infty} \| \re (e^{i \theta} \mathbf{X}^{\rm e}_n) \| = \| x^{\rm e}_n \|. 
\end{equation} 

By \eqref{conv to disc CGLZ14}, it suffices to show that $\| x^{\rm e}_n \|$ is independent of $\theta$ and $\tau$. 
For a free random variable $a$, it is convenient to adopt the notation 
\[
a^{\epsilon(\cdot)} =
\begin{cases}
a, & \text{if } \epsilon(\cdot)=\cdot,
\smallskip 
\\
a^{*}, & \text{if } \epsilon(\cdot)=* .
\end{cases}
\] 
Let $\phi$ denote the tracial state on the underlying non-commutative probability space.
Using the definition \eqref{def of uk and xn e} of $x_n^{\rm e}$ and expanding the $m$-th power, we obtain  
\begin{align} \label{expansion of phi xn power}
\phi \big( (x^{\rm e}_n)^m \big)
&= \phi \bigg( \Big(
\frac{e^{i \theta} \prod_{k=1}^{n} u_k + e^{-i \theta} \prod_{k=n}^{1} u_k^*}{2} \Big)^m \bigg)  = \frac{1}{2^m} 
\sum_{\epsilon:[m] \to \{\cdot, *\}}
(e^{i \theta})^{\Delta(\epsilon)}
\,\phi\big(u^{\epsilon(1)} \cdots u^{\epsilon(m)}\big),
\end{align} 
where
\[
\Delta(\epsilon)
: =
\big|\{k:\epsilon(k)=\cdot\}\big|
-
\big|\{k:\epsilon(k)=*\}\big|.
\]
Here the function $\epsilon:[m]\to\{\cdot,*\}$ records the choice of terms in the expansion of the product: for each $j\in[m]$, the value $\epsilon(j)=\cdot$ corresponds to selecting the factor $e^{i\theta}\prod_{k=1}^n u_k$, while $\epsilon(j)=*$ corresponds to selecting $e^{-i\theta}\prod_{k=n}^1 u_k^*$. We note that here $\prod_{k=n}^1 u_k^* = u_n^* u_{n-1}^* \cdots u_1^*$, and the notation $\prod_{k=n}^1$ is used to emphasise the order of the products. 

To evaluate the expectation on the right-hand side of \eqref{expansion of phi xn power}, we use the moment--cumulant formula of free probability. 
More precisely, if $NC(k)$ denotes the set of non-crossing partitions of $\{1,\dots,k\}$ and 
$\kappa_j$ denotes the $j$-th free cumulant with respect to $\phi$, then for any 
$a_1,\dots,a_k$ we have
\[
\phi(a_1 \cdots a_k)
=
\sum_{\pi \in NC(k)}
\prod_{V \in \pi}
\kappa_{|V|}(a_{i_1},\dots,a_{i_{|V|}}),
\]
where for a block $V=\{i_1,\dots,i_{|V|}\}$ we write
\(
\kappa_{|V|}(a_{i_1},\dots,a_{i_{|V|}})
\)
for the corresponding free cumulant.
Applying this identity to \eqref{expansion of phi xn power}, we have 
\begin{align*}
\phi \big( (x^{\rm e}_n)^m \big)
=
\frac{1}{2^m}
\sum_{\epsilon:[m] \to \{\cdot, *\}}
(e^{i \theta})^{\Delta(\epsilon)}
\sum_{\pi \in NC(mn)}
\prod_{\substack{V \in \pi \\ V=\{k_1,\dots,k_j\}}}
\kappa_j( u_{k_1}^{\epsilon(i_1)}, \dots, u_{k_j}^{\epsilon(i_j)} ). 
\end{align*}
Note that the block $V=\{k_1,\dots,k_j\}$ refers to positions in the word of length $mn$ obtained by expanding
\begin{equation} \label{4.6}
\prod_{l=1}^{m}\Big(\prod_{k=1}^{n}u_k\Big)^{\epsilon(l)}.
\end{equation} 
More precisely, $k_1,\dots,k_j$ indicate the locations in this word where the entries
$u_{k_1}^{\epsilon(i_1)},\dots,u_{k_j}^{\epsilon(i_j)}$ appear.
In particular, each position corresponds to one of the generators $u_1,\dots,u_n$ or their adjoints, and therefore the collection of generator indices occurring in the entire word forms the multiset
\[
\{1,\dots,1,\,2,\dots,2,\,\dots,\,n,\dots,n\},
\]
where each element $1,\dots,n$ appears exactly $m$ times.

Since the variables $u_k$ are centered and semicircular cumulants vanish for orders different from two, we have $\kappa_j(u_{k_1}^{\epsilon(i_1)}, \dots, u_{k_j}^{\epsilon(i_j)}) = 0$ for $j \not=2 $. 
Consequently, only pairings contribute to the moment expansion, and therefore
\begin{equation} \label{moment of product of eGinibre}
\phi \big( (x^{\rm e}_n)^m \big)
=
\frac{1}{2^m}
\sum_{\epsilon:[m]\to\{\cdot,*\}}
(e^{i\theta})^{\Delta(\epsilon)}
\sum_{\pi\in NC_2^{\epsilon}(mn)}
\prod_{\substack{V\in\pi\\ V=\{p,q\}}}
\kappa_2(u_{p}^{\epsilon(i)},u_{q}^{\epsilon(j)}).
\end{equation} 
Here $NC_2^{\epsilon}(mn)$ denotes the set of non-crossing pairings of the word of length $mn$ obtained from the expansion of \eqref{4.6}. 
As before, we slightly abuse notation when writing $p$ and $q$: for a block $V=\{p,q\}$ they represent the positions in this length-$mn$ word, whereas in $u_p$ and $u_q$ they indicate the indices of the free generators appearing at those positions. 

Since $\kappa_2(u_{p}^{\epsilon(i)}, u_{q}^{\epsilon(j)}) = 0$ whenever $p \neq q$, only those non-crossing pairings in $NC_2^{\epsilon}(mn)$ contribute in which $u_{p}^{\epsilon(i)}$ and $u_{q}^{\epsilon(j)}$ are paired if and only if $p=q$.
Therefore, each generator $u_k$ can only be paired with another occurrence of the same generator.
Since the subscript $k=1,\dots,n$ appears exactly $m$ times in the length-$mn$ word, this is possible only when $m$ is even. 
 
Suppose that for some $k$, the elements $u_k^{\epsilon(i)}$ and $u_k^{\epsilon(j)}$ are paired while $\epsilon(i)=\epsilon(j)$. 
Without loss of generality, assume that $\epsilon(i)=\cdot$. 
Consider the word of length $mn$ appearing in the expansion of $(x_n^{\rm e})^m$. Here the $u_k$ in the first block $(u_1 \cdots u_k \cdots u_n)$ is paired with the $u_k$ in the second block $(u_1 \cdots u_k \cdots u_n)$.
Between these two blocks, the word consists of a concatenation of blocks of the form
\emph{dot blocks} $(u_1\cdots u_n)$ or \emph{star blocks} $(u_n^*\cdots u_1^*)$.
Let $a$ denote the total number of such blocks occurring between the two blocks containing the paired entries; see Figure~\ref{Fig_pairing with a blocks} for an illustration.

\begin{figure}[h]
\centering
\begin{tikzpicture}[
    font=\small,
    every node/.style={inner sep=0pt, outer sep=0pt},
    mainblock/.style={draw, rectangle, minimum width=2.4cm, minimum height=0.95cm},
    subblock/.style={draw, rectangle, minimum width=2.0cm, minimum height=0.95cm},
    pairarc/.style={thick},
    blocklabel/.style={font=\small},
    scale=0.9
]

\node[mainblock] (B1) at (0,0) {};
\node at ($(B1.center)+(-0.9,0)$) {$u_1$};
\node at ($(B1.center)+(-0.45,0)$) {$\cdots$};
\node (uk1) at ($(B1.center)+(0,0)$) {$u_k$};
\node at ($(B1.center)+(0.45,0)$) {$\cdots$};
\node at ($(B1.center)+(0.9,0)$) {$u_n$};

\node[subblock] (M1) at (3.2,0) {};
\node at (M1.center) {$u_1\cdots u_n$};

\node[subblock] (M2) at (6.2,0) {};
\node at (M2.center) {$u_n^*\cdots u_1^*$};

\node at (8.2,0) {$\cdots$};
\node at (-2,0) {$\cdots$};
\node at (15.2,0) {$\cdots$};

\node[subblock] (M3) at (10.2,0) {};
\node at (M3.center) {$u_1\cdots u_n$};

\node[mainblock] (B2) at (13.2,0) {};
\node at ($(B2.center)+(-0.9,0)$) {$u_1$};
\node at ($(B2.center)+(-0.45,0)$) {$\cdots$};
\node (uk2) at ($(B2.center)+(0,0)$) {$u_k$};
\node at ($(B2.center)+(0.45,0)$) {$\cdots$};
\node at ($(B2.center)+(0.9,0)$) {$u_n$};

\node[blocklabel] at ($(M1.north)+(0,0.45)$) {dot block};
\node[blocklabel] at ($(M2.north)+(0,0.45)$) {star block};
\node[blocklabel] at ($(M3.north)+(0,0.45)$) {dot block};

\draw[pairarc]
($(uk1.north)+(0,0.1)$)
.. controls ($(uk1.north)+(0,1.75)$) and ($(uk2.north)+(0,1.75)$) ..
($(uk2.north)+(0,0.1)$);

\draw[decorate,decoration={brace,mirror,amplitude=5pt}]
($(M1.south west)+(-0.25,-0.45)$) -- ($(M3.south east)+(0.25,-0.45)$)
node[midway,yshift=-0.45cm] {$a$ intermediate blocks};

\end{tikzpicture}
\caption{Pairing of two occurrences of $u_k$. 
The segment between them consists of $a$ intermediate blocks, each being either a dot block $(u_1\cdots u_n)$ or a star block $(u_n^*\cdots u_1^*)$.}
\label{Fig_pairing with a blocks}
\end{figure}

Now consider the subsequence of the elements $u_{p}^{\epsilon(s)}$ lying between the two paired occurrences of $u_k$. 
These elements must themselves form a non-crossing pairing such that $u_{p}^{\epsilon(s)}$ and $u_{q}^{\epsilon(t)}$ are paired only if $p=q$.
In this subsequence the generator $u_k$ appears exactly $a$ times, hence $a$ must be even. 
On the other hand, for any $p\neq k$, the generator $u_p$ appears exactly $a+1$ times, which must also be even for such a pairing to exist. 
This is impossible, yielding a contradiction (here we use $n\ge2$).


Therefore the only non-crossing pairings $\pi\in NC_2^{\epsilon}(mn)$ that contribute are those for which
\[
u_{p}^{\epsilon(s)} \text{ and } u_{q}^{\epsilon(t)} \text{ are paired} \qquad \textup{if and only if} \qquad
p=q \text{ and } \{\epsilon(s),\epsilon(t)\}=\{\cdot,*\}.
\]
We call such a pairing a \emph{good pairing with respect to $\epsilon$}.

 
On the other hand, note that 
\begin{align} \label{2nd cumulant of u_k}
    \kappa_2(u_k, u_k^{*}) &= \kappa_2 \Big( \sqrt{\frac{1+\tau}{2}} s_k + i \sqrt{\frac{1-\tau}{2}} t_k, \sqrt{\frac{1+\tau}{2}} s_k - i \sqrt{\frac{1-\tau}{2}} t_k \Big)  = \frac{1+\tau}{2} + \frac{1-\tau}{2} = 1.
\end{align} 
Indeed, this is the position where the $\tau$-dependence disappears. 

For a good pairing with respect to $\epsilon$ to exist, the number of \emph{dot blocks} must equal the number of \emph{star blocks}; equivalently,
$\Delta(\epsilon)=0.$
Applying this condition together with \eqref{2nd cumulant of u_k} to \eqref{moment of product of eGinibre}, we obtain
\begin{align} \label{tau invariant m-th moment of product eGinibre}
\begin{split} 
\phi \big((x^{\rm e}_n)^m\big)
&=
\frac{1}{2^m}
\sum_{\substack{\epsilon:[m]\to\{\cdot,*\}\\ \Delta(\epsilon)=0}}
\sum_{\pi\in NC_2^{\epsilon,*}(mn)} 1
=
\frac{1}{2^m}
\sum_{\pi\in NC_2^{*}(mn)} 1
=
\frac{1}{2^m}\,\big|NC_2^{*}(mn)\big|. 
\end{split}
\end{align}
Here $NC_2^{\epsilon,*}(mn)$ denotes the set of good pairings contained in 
$NC_2^{\epsilon}(mn)$, while 
\begin{equation} \label{def of NC2 * mn}
NC_2^{*}(mn)
:=
\bigcup_{\substack{\epsilon:[m]\to\{\cdot,*\}\\ \Delta(\epsilon)=0}}
NC_2^{\epsilon,*}(mn)
\end{equation} 
denotes the union of the sets of good pairings over all functions 
$\epsilon:[m]\to\{\cdot,*\}$ satisfying $\Delta(\epsilon)=0$.

By \eqref{tau invariant m-th moment of product eGinibre}, the moment 
$\phi\big((x^{\rm e}_n)^m\big)$ is independent of both $\theta$ and $\tau$. 
Consequently, $\|\re(e^{i\theta}\mathbf{X}^{\rm e}_n)\|$ is also independent of $\theta$ and $\tau$. This completes the proof. 
\end{proof}

By Lemma~\ref{tau invariance of the num range of product eGinibre}, it suffices to treat the case $\tau=\theta=0$. 
For circular elements $c_1, c_2, \dots, c_n$, define
\begin{equation} \label{def of xn prod of cj's}
    x_n := \re\Big( \prod_{k=1}^n c_k \Big).
\end{equation} 
Then, as before, by \cite[Theorem~A]{HT05} (see also \cite[Theorem~2.2]{FHS22}) we have 
\begin{equation}
    \lim_{N \to \infty} \| \re(\mathbf{X}_n) \| = \| x_n \|,
\end{equation}
and \eqref{tau invariant m-th moment of product eGinibre} gives
\begin{equation} \label{moment of xn in terms of NC2}
    \phi \big( x_n^m \big) = \frac{1}{2^m} \big| NC_2^{*}(mn) \big|.
\end{equation}
Note that in \eqref{def of NC2 * mn}, $\big| NC_2^{*}(mn) \big|=0$ unless $m$ is even, since the numbers of dot blocks and star blocks must agree. 
Accordingly, we write $m=2t$ and define
\begin{equation} \label{def of A0}
    A_0(t) := \big| NC_2^{*}(2nt) \big|,
\end{equation}
with the convention that $A_0(0)=1$.  
Furthermore, for $1 \le i \le n-1$, we define $A_i(t)$ to be the number of good pairings of the subsequence between $c_{n-i+1}$ and $c_{n-i+1}^*$ under the condition that the adjacent entries $c_{n-i}$ and $c_{n-i}^*$ are paired; see Figure~\ref{fig:Ai-pairing} for an illustration. 

\begin{figure}[h]
\centering
\begin{tikzpicture}[every node/.style={font=\small},scale=0.9]

\node (c1) at (0,0) {$c_1$};
\node at (0.7,0) {$\cdots$};
\node (ci) at (1.5,0) {$c_{n-i}$};
\node at (2.3,0) {$\cdots$};
\node (cn) at (3.1,0) {$c_n$};

\node at (6.5,0) {$\cdots$};

\node (cnstar) at (9.9,0) {$c_n^*$};
\node at (10.7,0) {$\cdots$};
\node (cistar) at (11.5,0) {$c_{n-i}^*$};
\node at (12.3,0) {$\cdots$};
\node (c1star) at (13.1,0) {$c_1^*$};

\draw (-0.4,-0.5) rectangle (3.5,0.5);
\draw (9.5,-0.5) rectangle (13.5,0.5);

\draw (1.0,-0.5) -- (1.0,0.5);
\draw (12.0,-0.5) -- (12.0,0.5);

\draw[thick]
($(ci.north)$)
.. controls ($(ci.north)+(0,1.5)$) and ($(cistar.north)+(0,1.5)$) ..
($(cistar.north)$);

\draw[decorate,decoration={brace,mirror,amplitude=6pt}]
(3.7,-0.9) -- (9.3,-0.9)
node[midway,yshift=-0.5cm] {$2t$ blocks};

\end{tikzpicture}
\caption{The configuration corresponding to $A_i(t)$.}
\label{fig:Ai-pairing}
\end{figure}

As illustrated in Figure~\ref{fig:Ai-pairing}, the entries $c_{n-i}$ and $c_{n-i}^*$ are paired, and there are $2t$ blocks between the blocks containing these paired elements. 
Since the number of dot blocks must coincide with the number of star blocks, the total number of such blocks must be even, which we denote by $2t$. 
Consequently, there are $2nt$ free random variables between these two blocks, and hence a total of $2nt+2i$ free random variables involved in the pairing.

In the case $t=0$, there is only a single good pairing, and therefore
\begin{equation} 
A_1(0)=A_2(0)=\cdots=A_{n-1}(0)=1.
\end{equation} 
Using a simple combinatorial argument, we derive recursive formulas for $A_i(t)$. 

\begin{lem}
We have  
\begin{equation}\label{recursive formula for A_0}
A_0(t+1)
= 2 \sum_{k=0}^{t} A_{n-1}(k)\,A_0(t-k) 
\end{equation}
and for $i=1,\dots,n-1,$ 
\begin{equation}\label{recursive formula for A_i}
    A_i(t+1)
    =
    A_{i-1}(t+1)
    +
    \sum_{k=0}^{t} A_{i-1}(k)
    \sum_{l=0}^{t-k} A_{n-i-1}(l)\,A_i(t-k-l). 
\end{equation} 
\end{lem}

\begin{proof}

We first derive the recursive formula for $A_0(t+1)$. 
The first block may be either a dot block or a star block. 
By symmetry, it suffices to consider the case where the first block is a dot block; we will multiply the final count by $2$ to account for the two possibilities.

Let the first random variable $c_1$ in the first block be paired with $c_1^*$ in the $(2k+2)$-th block for some $k=0,1,\dots,t$. 
Note that $c_1$ cannot be paired with a $c_1^*$ in a $(2k+1)$-st block, since in that case there would be $2k-1$ blocks between the two paired entries, and hence the numbers of dot and star blocks between them could not be equal.
Thus we consider the configuration illustrated in Figure~\ref{fig:recursive-A0}.

\begin{figure}[h]
\centering
\begin{tikzpicture}[
    font=\small,
    every node/.style={inner sep=0pt, outer sep=0pt},
    mainblock/.style={draw, rectangle, minimum width=2.6cm, minimum height=0.95cm},
    pairarc/.style={thick},scale=0.9
]

\node[mainblock] (B1) at (0,0) {};
\node (c1) at ($(B1.center)+(-0.75,0)$) {$c_1$};
\node at ($(B1.center)+(0,0)$) {$\cdots$};
\node at ($(B1.center)+(0.75,0)$) {$c_n$};

\node at (4.0,0) {$\cdots$};

\node[mainblock] (B2) at (8.0,0) {};
\node at ($(B2.center)+(-0.75,0)$) {$c_n^*$};
\node at ($(B2.center)+(0,0)$) {$\cdots$};
\node (c1star) at ($(B2.center)+(0.75,0)$) {$c_1^*$};

\node at (12.0,0) {$\cdots$};

\draw[pairarc]
($(c1.north)+(0,0.1)$) .. controls ($(c1.north)+(0,1.75)$) and ($(c1star.north)+(0,1.75)$) .. ($(c1star.north)+(0,0.1)$);

\draw[decorate,decoration={brace,mirror,amplitude=6pt}]
(1.5,-0.9) -- (6.5,-0.9)
node[midway,yshift=-0.5cm] {$2k$ blocks};

\draw[decorate,decoration={brace,mirror,amplitude=6pt}]
(9.7,-0.9) -- (14.3,-0.9)
node[midway,yshift=-0.5cm] {$2(t-k)$ blocks};

\draw[decorate,decoration={brace,mirror,amplitude=6pt}]
(-1.2,-1.8) -- (8.4,-1.8)
node[midway,yshift=-0.6cm] {(1)};

\draw[decorate,decoration={brace,mirror,amplitude=6pt}]
(9.7,-1.8) -- (14.3,-1.8)
node[midway,yshift=-0.6cm] {(2)};

\end{tikzpicture}
\caption{Configuration contributing to the recursion for $A_0(t+1)$.}
\label{fig:recursive-A0}
\end{figure}

The segment (1) contains $2k$ blocks (together with the paired entries $c_{1}$ and $c_{1}^*$ adjacent to them.)  
Hence the number of good pairings in this region is $A_{n-1}(k)$. 
On the other hand, the segment (2) contains $2(t-k)$ blocks, which contributes $A_0(t-k)$ admissible pairings.  

Combining these contributions and summing over all possible $k=0,1,\dots,t$, we obtain \eqref{recursive formula for A_0}, where the factor $2$ accounts for the two possible choices of the first block (dot or star).

Next, we derive a recursive formula for $A_i(t+1)$, where $1 \le i \le n-1$. 
Let the first occurrence of $c_{n-i}$ be paired with $c_{n-i}^*$ in the $(2k+2)$-th block. 
As before, it cannot be paired with a copy of $c_{n-i}^*$ in a $(2k+1)$-st block, since the numbers of dot and star blocks in between would then be unequal. 
Here $k$ may take any value in $\{0,1,\dots,t+1\}$. 
If $k=t+1$, then the number of good pairings is simply $A_{i-1}(t+1)$.

For the remaining cases $k=0,1,\dots,t$, we are led to the configuration shown in Figure~\ref{fig:recursive-Ai}.

\begin{figure}[h]
\centering
\begin{tikzpicture}[every node/.style={font=\small},scale=0.9]


\node (c1) at (0,0) {$c_1$};
\node (dots1) at (0.6,0) {$\cdots$};
\draw (1.1,-0.5)--(1.1,0.5);
\node (cni) at (1.5,0) {$c_{n-i}$};
\node (cni1) at (2.4,0) {$c_{n-i+1}$};
\node (dots2) at (3.2,0) {$\cdots$};
\node (cn) at (3.9,0) {$c_n$};

\draw (-0.35,-0.5) rectangle (4.2,0.5);

\node (mid1) at (5.3,0) {$\cdots$};

\node (cnstar) at (6.6,0) {$c_n^*$};
\node (dots3) at (7.3,0) {$\cdots$};
\node (cni1s) at (8.0,0) {$c_{n-i+1}^*$};
\node (cniS) at (8.9,0) {$c_{n-i}^*$};
\node (dots4) at (9.6,0) {$\cdots$};
\node (c1s) at (10.3,0) {$c_1^*$};

\draw (6.2,-0.5) rectangle (10.6,0.5);

\node (mid2) at (11.4,0) {$\cdots$};

\node (cnstar2) at (12.6,0) {$c_n^*$};
\node (dots5) at (13.3,0) {$\cdots$};
\node (cni1s2) at (14.0,0) {$c_{n-i+1}^*$};
\node (cniS2) at (14.9,0) {$c_{n-i}^*$};

\draw (15.4,-0.5)--(15.4,0.5);

\node (dots6) at (16.1,0) {$\cdots$};
\node (c1s2) at (16.8,0) {$c_1^*$};

\draw (12.2,-0.5) rectangle (17.1,0.5);


\draw[thick]
($(cni1.north)+(0,0.1)$)
.. controls ($(cni1.north)+(0,1.5)$) and ($(cni1s.north)+(0,1.5)$) ..
($(cni1s.north)+(0,0.1)$);

\draw[thick]
($(cni.north)+(0,0.1)$)
.. controls ($(cni.north)+(0,2.5)$) and ($(cniS2.north)+(0,2.5)$) ..
($(cniS2.north)+(0,0.1)$);

\draw[decorate,decoration={brace,mirror,amplitude=6pt}]
(4.3,-0.9) -- (6.1,-0.9)
node[midway,yshift=-0.5cm] {$2k$ blocks};

\draw[decorate,decoration={brace,mirror,amplitude=6pt}]
(10.7,-0.9) -- (12.1,-0.9)
node[midway,yshift=-0.5cm] {$2(t-k)$ blocks $+\,1$ dot block};

\draw[decorate,decoration={brace,mirror,amplitude=6pt}]
(2.2,-1.8) -- (8.4,-1.8)
node[midway,yshift=-0.6cm] {(3)};

\draw[decorate,decoration={brace,mirror,amplitude=6pt}]
(8.9,-1.8) -- (14.3,-1.8)
node[midway,yshift=-0.6cm] {(4)};

\draw[dotted] (1.9,0.65) -- (1.9,-2.0);
\draw[dotted] (8.55,0.65) -- (8.55,-2.0);
\draw[dotted] (14.5,0.65) -- (14.5,-2.0);

\coordinate (fourL) at (8.9,-1.8);
\coordinate (fourR) at (14.3,-1.8);


\begin{scope}[yshift=-5cm, xshift=3cm]

\node (cniS) at (0,0) {$c_{n-i}^*$};
\node (dots1) at (0.7,0) {$\cdots$};
\node (c1s) at (1.5,0) {$c_1^*$};

\draw[dotted] (-0.35,-0.5) -- (-0.35,0.5);
\draw (-0.35,0.5) -- (1.8,0.5);
\draw (-0.35,-0.5) -- (1.8,-0.5);
\draw (1.8,-0.5) -- (1.8,0.5);

\node (mid1) at (3.0,0) {$\cdots$};

\node (c1) at (4.3,0) {$c_1$};
\node (dots2) at (5.0,0) {$\cdots$};
\node (cni) at (5.8,0) {$c_{n-i}$};
\node (cni1) at (6.7,0) {$c_{n-i+1}$};
\node (dots3) at (7.5,0) {$\cdots$};
\node (cn) at (8.2,0) {$c_n$};

\draw (4.0,-0.5) rectangle (8.5,0.5);

\node (mid2) at (9.5,0) {$\cdots$};

\node (cnstar) at (10.7,0) {$c_n^*$};
\node (dots4) at (11.4,0) {$\cdots$};
\node (cni1s) at (12.2,0) {$c_{n-i+1}^*$};

\draw (10.3,-0.5) -- (10.3,0.5);
\draw (10.3,0.5) -- (12.7,0.5);
\draw (10.3,-0.5) -- (12.7,-0.5);
\draw[dotted] (12.7,-0.5) -- (12.7,0.5);

\draw[thick]
($(cni.north)+(0,0.1)$)
.. controls ($(cni.north)+(0,1.5)$) and ($(cniS.north)+(0,1.5)$) ..
($(cniS.north)+(0,0.1)$);

\draw[decorate,decoration={brace,mirror,amplitude=6pt}]
(1.9,-0.9) -- (3.9,-0.9)
node[midway,yshift=-0.5cm] {$2l$ blocks};

\draw[decorate,decoration={brace,mirror,amplitude=6pt}]
(8.6,-0.9) -- (10.2,-0.9)
node[midway,yshift=-0.5cm] {$2(t-k-l)$ blocks};

\draw[decorate,decoration={brace,mirror,amplitude=6pt}]
(-0.05,-1.7) -- (6.0,-1.7)
node[midway,yshift=-0.5cm] {(5)};

\draw[decorate,decoration={brace,mirror,amplitude=6pt}]
(6.5,-1.7) -- (12.5,-1.7)
node[midway,yshift=-0.5cm] {(6)};

\draw[dotted] (-0.35,1.5) -- (-0.35,-2.0);
\draw[dotted] (6.2,0.65) -- (6.2,-2.0);
\draw[dotted] (12.7,1.5) -- (12.7,-2.0);

\coordinate (botL) at (-0.35,1.5);
\coordinate (botR) at (12.7,1.5);

\end{scope}


\draw[dashed, blue] (8.55,-2.0) -- (botL);
\draw[dashed, blue] (14.5,-2.0) -- (botR);

\end{tikzpicture} \caption{Configuration contributing to the recursion for $A_i(t+1)$.}
\label{fig:recursive-Ai}
\end{figure}

The region labelled (3) contributes $A_{i-1}(k)$. 
The region labelled (4) is further decomposed according to the position of the first pairing of $c_{n-i}$, which yields the regions (5) and (6).  
More precisely, for each $l=0,1,\dots,t-k$, the region (5) contributes $A_{n-i-1}(l)$, while the region (6) contributes $A_i(t-k-l)$. 
Therefore, we obtain \eqref{recursive formula for A_i}, which completes the proof.  
\end{proof}

We are now ready to prove Theorem~\ref{Thm_products eGinUE}.

\begin{proof}[Proof of Theorem~\ref{Thm_products eGinUE}]

Recall that $A_i$'s are defined in \eqref{def of A0} and the text below. 
For $0 \le i \le n-1$, we consider the generating function 
\begin{equation} \label{def of moment generating for Ai's}
  p_i(x):=\sum_{s=0}^{\infty} A_i(s)x^s. 
\end{equation}
Then by \eqref{recursive formula for A_0} and \eqref{recursive formula for A_i}, we have 
\begin{align}
\label{recursive formula for p_0}
    p_0(x)  -1& =  2x\,p_0(x)p_{n-1}(x)
\end{align}
and 
\begin{align}
     p_j(x) -p_{j-1}(x)& = x\,p_{j-1}(x)p_j(x)p_{n-j-1}(x), \label{recursive formula for p_j} 
\end{align}
where $1\le j\le n-1.$ 

Replacing $j$ by $n-j$ in \eqref{recursive formula for p_j}, we have
\begin{equation}\label{recursive formula for p_{n-j}}
    p_{n-j}(x)-p_{n-j-1}(x)= x\,p_{n-j-1}(x)p_{n-j}(x)p_{j-1}(x).  
\end{equation}
Dividing \eqref{recursive formula for p_j} and \eqref{recursive formula for p_{n-j}}, we obatin
\begin{equation}\label{ratio of p_j (1)}
    \frac{p_j(x)}{p_{j-1}(x)}=\frac{p_{n-j}(x)}{p_{n-j-1}(x)},
    \qquad 1\le j\le n-1.
\end{equation}
Similarly, for $1\le j\le n-2$, replacing $j$ by $n-j-1$ in \eqref{recursive formula for p_j}, we obtain
\begin{equation}
\frac{p_{j}(x)}{p_{j-1}(x)}=\frac{p_{n-j-1}(x)}{p_{n-j-2}(x)},
\qquad 1\le j\le n-2.
\end{equation} 
Combining the above, we obtain that the ratio $p_j(x)/p_{j-1}(x)$ is constant in $j$. By denoting 
\begin{equation} 
r(x):=\frac{p_1(x)}{p_0(x)},
\end{equation} 
we have 
\begin{equation}\label{pj geometric}
    p_j(x)=r(x)^j p_0(x), \qquad 0\le j\le n-1.
\end{equation}

We now determine $r(x)$ in terms of $p_0(x)$. Note that by \eqref{recursive formula for p_0} and \eqref{pj geometric}, we have
\begin{equation}\label{p0 eq 1}
    p_0(x)-1=2x\,r(x)^{\,n-1}p_0(x)^2.
\end{equation}
On the other hand, taking $j=1$ in \eqref{recursive formula for p_j}, it follows from \eqref{pj geometric} that 
\begin{equation}\label{p0 eq 2}
    (r(x)-1)p_0(x)=x\,r(x)^{\,n-1}p_0(x)^3.
\end{equation}
Comparing \eqref{p0 eq 1} and \eqref{p0 eq 2}, we deduce that 
\begin{equation}\label{r in terms of p0}
    r(x)=\frac{p_0(x)+1}{2}.
\end{equation}
Substituting this into \eqref{p0 eq 1}, we arrive at
\begin{equation}\label{polynomial of p_0}
    \frac{p_0(x)-1}{2} = x\Big(\frac{p_0(x)+1}{2}\Big)^{n-1}p_0(x)^2.
\end{equation}

Recall that $x_n$ is given by \eqref{def of xn prod of cj's}. We now consider the moment generating function 
\begin{equation}
M(z)  := 1 + \sum_{m=1}^{\infty} \phi((x_n)^m) z^m . 
\end{equation}
Then by \eqref{moment of xn in terms of NC2}, \eqref{def of A0} and \eqref{def of moment generating for Ai's}, we have  
\begin{align}
    M(z) = 1 + \sum_{t=1}^{\infty} \frac{A_0(t)}{2^{2t}} z^{2t}  = p_0 \Big( \frac{z^2}{4} \Big) . 
\end{align}
Then the Cauchy transform $G(z)$ of the distribution of the free random variable $x_n$ is given by 
\begin{equation} \label{relationship of G and P_0}
    G(z) = \frac{1}{z} M\Big( \frac{1}{z} \Big) = \frac{1}{z} P_0 \Big(\frac{1}{4z^2} \Big) .
\end{equation}
Then by \eqref{polynomial of p_0}, we obtain the algebraic equation 
\begin{equation}
    2^n (zG(z) - 1) = G(z)^2 (zG(z)+1)^{n-1} .
\end{equation}

Let
\begin{equation}
    P(z,w) := w^2(zw+1)^{n-1} - 2^n(zw-1).
\end{equation}
Then $G(z)$ is determined by the equation $P(z,G(z))=0$. 
As before, the endpoints of the support correspond to branch points of this algebraic relation, and hence are obtained by solving
\begin{equation}
    P(z,w)=0, \qquad \partial_w P(z,w)=0 .
\end{equation}

From $P(z,w)=0$ we obtain
\begin{equation}
    w^2(zw+1)^{n-1}=2^n(zw-1),
\end{equation}
while $\partial_w P(z,w)=0$ yields
\begin{equation}
    (zw+1)^{n-2}\big((n+1)zw^2+2w\big)=2^n z .
\end{equation}
Solving this system of equations gives
\begin{equation}
    z = R_n
    =
    \sqrt{
        \frac{\rho_n^2(\rho_n+1)^{n-1}}
        {2^n(\rho_n-1)}
    },
    \qquad
    \rho_n := \frac{1+\sqrt{1+8/n}}{2}.
\end{equation}
This determines the endpoint of the support, which is the constant $R_n$ given in \eqref{def of Rn radius products}, and completes the proof. 
\end{proof}

We end this section by showing that, in the Ginibre matrix case, one may indeed allow repeated factors in the definition of the product model, and that the conclusion of Theorem~\ref{Thm_products eGinUE} remains valid.

\begin{prop} \label{Prop_index invariance for products of GinUEs}
Let $n \ge 1$ and $X_1, X_2, \cdots, X_n$ be $N \times N$ i.i.d. Ginibre matrices. For all $I: [n] \to [n]$, i.e., $i_1, \cdots, i_n \in \{1, \cdots, n\}$, denote $\mathbf{X}_I := X_{i_1}X_{i_2} \cdots X_{i_n}$. Then,
\begin{equation}
    \lim_{N \to \infty} d_H(W(\mathbf{X}_I), \mathbb{D}(R_n)) = 0 .
\end{equation}
almost surely, where $R_n$ is given by \eqref{def of Rn radius products}. 
\end{prop}
\begin{proof}
Let $c_1,c_2,\dots,c_n$ be circular elements. For an index set $I=(i_1,\dots,i_n)$, define
\[
c_I := c_{i_1}c_{i_2}\cdots c_{i_n}, 
\qquad 
x_I := \re(c_I).
\]
Then, as in the proof of Lemma~\ref{tau invariance of the num range of product eGinibre}, we have
\begin{equation}
\lim_{N\to\infty}\|\re(\mathbf{X}_I)\|=\|x_I\|.
\end{equation}
Moreover,
\begin{equation}
\phi\big((x_I)^m\big)
=
\frac{1}{2^m}
\sum_{\epsilon:[m]\to\{\cdot,*\}}
(e^{i\theta})^{\Delta(\epsilon)}
\sum_{\pi\in NC_2^{\epsilon}(mn)}
\prod_{\substack{V\in\pi\\ V=\{i_p,i_q\}}}
\kappa_2\!\left(c_{i_p}^{\epsilon(i)},c_{i_q}^{\epsilon(j)}\right).
\end{equation}

Recall that circular elements satisfy
\[
\kappa_2\!\left(c_{i_p}^{\epsilon(i)},c_{i_q}^{\epsilon(j)}\right)\neq 0
\quad\text{if and only if}\quad
i_p=i_q
\ \text{and}\
\{\epsilon(i),\epsilon(j)\}=\{\cdot,*\}.
\]
Thus, in any contributing pairing, each circular element must be paired with the adjoint of itself.

Consider a good pairing in which $c_{i_p}$ is paired with $c_{i_q}^*$ ($1\le p,q\le n$). Suppose that between these two elements there are $t$ dot blocks and $s$ star blocks.
Then the number of circular elements between $c_{i_p}$ and $c_{i_q}^*$ is
$n-p+nt,$ while the number of adjoint circular elements between them is $
n-q+ns$. 
Since each circular element must pair with the adjoint of itself, these two numbers must coincide, and hence
\begin{equation}
n-p+nt = n-q+ns.
\end{equation}
This implies that $n\mid(p-q)$. As $1\le p,q\le n$, it follows that $p=q$. 
We have therefore shown that in any good pairing, if $c_{i_p}$ is paired with $c_{i_q}^*$, then necessarily $p=q$.

Now let $J=(j_1,\dots,j_n)$ be another index set. Consider the bijection
$c_{i_k}\leftrightarrow c_{j_k}$, $(k=1,\dots,n$).
In this correspondence we regard the symbols $c_{i_p}$ and $c_{i_q}$ as distinct letters whenever $p\neq q$, even if $i_p=i_q$. Under this identification, every good pairing of length $mn$ arising from the expansion of $c_I$ and $c_I^*$ corresponds bijectively to a good pairing arising from the expansion of $c_J$ and $c_J^*$, and vice versa. 
Consequently, in the Ginibre case the numerical range does not depend on the index set.
\end{proof}

\appendix

\section{Geometry of numerical range of non-Hermitian Wishart matrix} \label{Appendix_non ellipse}

In this appendix, we expand upon Remark~\ref{Rem_non ellipse} by discussing that the numerical range of the non-Hermitian Wishart matrix in Theorem~\ref{Thm_NWishart} is not an ellipse, despite its apparent resemblance to one.

Let $\xi_{+} > \xi_{-}$ denote the two real roots of $D(x)$ defined in~\eqref{quartic polynomial of x}, and set 
\begin{equation} \label{def of c and A}
c = \frac{\xi_+ + \xi_-}{2}, \qquad A = \frac{\xi_+ - \xi_-}{2}.
\end{equation} 
Suppose that $\widetilde{E}(\alpha,\tau)$ is an ellipse. 
By identifying its major and minor axes, it must then take the form
\begin{equation} \label{def of tilde E ellipse assumption}
 \Big\{ (x,y)\in \R^2 : \Big(\frac{x-c}{A} \Big)^2 + \Big(\frac{y}{\sqrt{1-\tau^2}B} \Big)^2 \le 1 \Big\}, 
\end{equation}
where $B$ is given by \eqref{def of B}.

\begin{figure}[h]
    \centering
    \includegraphics[width=0.8\textwidth]{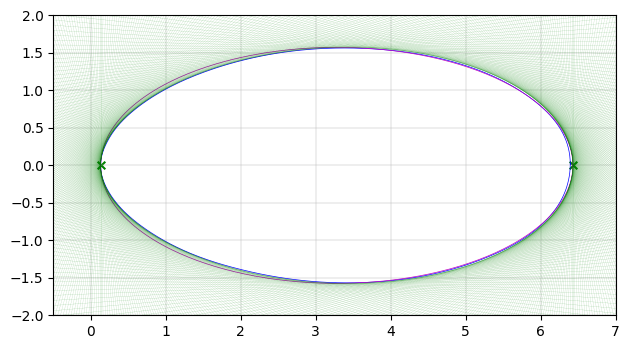}
    \caption{The simulated numerical range of the non-Hermitian Wishart matrix (blue solid curve), compared with the theoretical result in Theorem~\ref{Thm_NWishart} (the region bounded by the green envelope), and with the elliptical ansatz~\eqref{def of tilde E ellipse assumption}. Here, $\alpha=2, \tau=0.8$ and $N=3000$.}
    \label{Fig_non ellipse}
\end{figure}

The numerical simulations in Figure~\ref{Fig_non ellipse} confirm that the empirical numerical range of the non-Hermitian Wishart ensemble agrees with the set defined in~\eqref{def of tilde E}, rather than with the elliptical approximation in~\eqref{def of tilde E ellipse assumption}. 
In Figure~\ref{Fig_non ellipse}, the numerically computed numerical range is shown by the blue curve, while the green lines represent the supporting half-spaces $H_\theta$. The intersection of these half-spaces (the white region bounded by their envelope) corresponds to the theoretical numerical range established in Theorem~\ref{Thm_NWishart}. For comparison, the purple curve depicts the ellipse constructed under the heuristic elliptical assumption~\eqref{def of tilde E ellipse assumption}. 
Although the elliptical approximation appears visually close at first glance, a noticeable discrepancy emerges near the boundary close to the leftmost point of the numerical range. 
In this region, the simulation aligns with our theoretical prediction in Theorem~\ref{Thm_NWishart}, rather than with the elliptical ansatz.

One can in fact show analytically that the ansatz~\eqref{def of tilde E ellipse assumption} is false. 
Indeed, suppose that $\widetilde{E}(\alpha,\tau)$ is of the elliptical form~\eqref{def of tilde E ellipse assumption}. 
Then, by~\eqref{def of tilde E}, its support function (cf.~\eqref{def of support function of ellipse}) is given by
$$
c \cos(\theta) \pm \sqrt{A^2 \cos^2(\theta) + (1-\tau^2) B^2 \sin^2(\theta)}.
$$ 
Consequently, these expressions should coincide with the two real roots of $D_\theta(x)=0$. 
Equivalently, $D_\theta(x)$ should admit the factorisation
\begin{align} \label{factorize D_theta(x)}
\begin{split} 
    D_{\theta}(x) &= \Big(x^2 - 2c \cos(\theta)x + (c^2-A^2) \cos^2(\theta) - (1-\tau^2)b \sin^2 (\theta)\Big) 
    \\
    & \quad \times \Big( 16(1-\tau^2\sin^2(\theta))x^2 + 32(c-\tau(\alpha+2))(1-\tau^2 \sin^2(\theta))\cos(\theta)x + d_{\theta} \Big),
\end{split}
\end{align}
where $b=B^2$ with $B$ given in \eqref{def of B} and 
\begin{align} \label{coeff of x^2 of D_theta(x)}
\begin{split} 
    d_{\theta} &= 16 \tau^4 \alpha^2 \cos^4(\theta) + 8 \tau^2 (1-\tau^2) (2\alpha^2 - 5\alpha - 6) \cos^2(\theta) + 4(1-\tau^2)^2(\alpha^2 - 8 \alpha - 11) 
    \\
    &\quad -16(c^2-A^2)(1-\tau^2\sin^2(\theta))\cos^2(\theta)+16(1-\tau^2)(1-\tau^2\sin^2(\theta))b\sin^2(\theta). 
\end{split}
\end{align} 
Here, $d_\theta$ is determined by matching the coefficients of $x^2$.

Comparing the remaining coefficients in~\eqref{factorize D_theta(x)} leads, after lengthy but straightforward computations, to algebraic constraints on the parameters $\alpha$ and $\tau$. 
Since these identities must hold for all $\theta$, they would force $\alpha$ and $\tau$ to satisfy a nontrivial algebraic relation. 
However, this is impossible, as $\tau \in [0,1]$ and $\alpha \ge 0$ are free parameters. 
This contradiction shows that the ansatz~\eqref{def of tilde E ellipse assumption} cannot be valid.

\bibliographystyle{abbrv}

\begin{thebibliography}{100}

\bibitem{AB10} G. Akemann and M. Bender, \emph{Interpolation between Airy and Poisson statistics for unitary chiral non-Hermitian random matrix ensembles}, J. Math. Phys. \textbf{51} (2010), 103524. 

\bibitem{AB12} G. Akemann and Z. Burda, \emph{Universal microscopic correlations for products of independent Ginibre matrices}, J. Phys. A \textbf{45} (2012), 465210. 

\bibitem{ABK21} G. Akemann, S.-S. Byun and N.-G. Kang, \emph{A non-Hermitian generalisation of the Marchenko–Pastur distribution: from the circular law to multi-criticality}, Ann. Henri Poincaré \textbf{22} (2021), 1035--1068.

\bibitem{ABN24} G. Akemann, S.-S. Byun and K. Noda, \emph{Pfaffian structure of the eigenvector overlap for the symplectic Ginibre ensemble}, Ann. Henri Poincaré (Online), https://doi.org/10.1007/s00023-025-01575-x, arXiv:2407.17935.  

\bibitem{ATTZ20} G. Akemann, R. Tribe, A. Tsareas and O. Zaboronski, 
\emph{On the determinantal structure of conditional overlaps for the complex Ginibre ensemble},  Random Matrices Theory Appl. \textbf{9} (2020), 2050015.

\bibitem{AK22} J. Alt and T. Kr\"{u}ger, \emph{Local elliptic law}, Bernoulli \textbf{28} (2022), no. 2, 886--909.

\bibitem{BS10} Z. Bai and J. W. Silverstein, \emph{Spectral analysis of large dimensional random matrices}, volume 20. Springer, 2010.

\bibitem{BSY88} Z. D. Bai, J. W. Silverstein and Y. Q. Yin, \emph{A note on the largest eigenvalue of a large-dimensional sample covariance matrix}, J. Multivariate Anal. \textbf{26} (1988), 166--168, 1988.

\bibitem{BY93} Z. D. Bai and Y. Q. Yin, \emph{Limit of the smallest eigenvalue of a large-dimensional sample covariance matrix}, Ann. Probab. \textbf{21} (1993), 1275--1294.

\bibitem{BC25} Z. Bao and G. Cipolloni, \emph{Numerical radius of non-Hermitian random matrices}, arXiv:2510.02667. 

\bibitem{BCEHK25} Z. Bao, G. Cipolloni, L. Erd\H{o}s, J. Henheik and O. Kolupaiev, \emph{Decorrelation transition in the Wigner minor process}, Probab. Theory Related Fields (to appear), arXiv:2503.06549.

\bibitem{BCEHK25a} Z. Bao, G. Cipolloni, L. Erd\H{o}s, J. Henheik and O. Kolupaiev, \emph{Law of fractional logarithm for random matrices}, arXiv:2503.18922. 

\bibitem{BNST17} S. Belinschi, M. A. Nowak R. Speicher and W. Tarnowski, \emph{Squared eigenvalue condition numbers and eigenvector correlations from the single ring theorem}, J. Phys. A {\bf 50} (2017), 105204.

\bibitem{Be10} M.Bender, \emph{Edge scaling limits for a family of non-Hermitian random matrix ensembles}, Probab.Theory Related Fields \textbf{147} (2010),241--271.

\bibitem{BBD23} M. Bhattacharjee, A. Bose and A. Dey, \emph{Joint convergence of sample cross-covariance matrices}, ALEA, Lat. Am. J. Probab. Math. Stat. \textbf{20} (2023), 395--423.

\bibitem{BD21} P. Bourgade and G. Dubach, \emph{The distribution of overlaps between eigenvectors of Ginibre matrices}, Probab. Theory Relat. Fields \textbf{177} (2020), 397--464.

\bibitem{BGNTW15} Z. Burda, J. Grela, M. A. Nowak, W. Tarnowski and P.
 Warchol, \emph{Unveiling the significance of eigenvectors in diffusing non-Hermitian matrices by identifying the underlying Burgers dynamics}, Nucl. Phys. B \textbf{897} (2015), 421--447.
 

\bibitem{BF25} S.-S.~Byun and P. J.~Forrester, \emph{Progress on the study of the Ginibre ensembles}, Springer Singapore, KIAS Springer Series in Mathematics (2025).

\bibitem{BF25a} S.-S.~Byun and P. J.~Forrester, \emph{Electrostatic computations for statistical mechanics and random matrix applications}, arXiv:2510.14334.

\bibitem{BN24} S.-S. Byun and K. Noda, \emph{Scaling limits of complex and symplectic non-Hermitian Wishart ensembles}, J. Approx. Theory \textbf{308} (2025), 106148.


\bibitem{CD07} M. Capitaine and C. Donati-Martin, \emph{Strong asymptotic freeness for Wigner and Wishart matrices}, Indiana Univ. Math. J. \textbf{56} (2007), 767--803.

\bibitem{CM98} J. T. Chalker and B. Mehlig, \emph{Eigenvector statistics in non-Hermitian random matrix ensembles}, Phys. Rev. Lett. \textbf{81} (1998), 3367--3370.

\bibitem{CJQ20} S. Chang, T. Jiang and Y. Qi, \emph{Eigenvalues of large chiral non-Hermitian random matrices}, J. Math. Phys. \textbf{61} (2020), 013508. 
 
\bibitem{Ch22} C. Charlier, \emph{Asymptotics of determinants with a rotation-invariant weight and discontinuities along circles}, Adv. Math. \textbf{408} (2022), 108600.

\bibitem{CGLZ14} B. Collins, P. Gawron, A. E. Litvak and K. Zyczkowski, \emph{Numerical range for random matrices}, J. Math. Anal. Appl. \textbf{418} (2014), 516--533.


\bibitem{CESX22} G. Cipolloni, L. Erd\H{o}s, D. Schr\"oder and Y. Xu, \emph{Directional extremal statistics for Ginibre eigenvalues}, J. Math. Phys. \textbf{63} (2022), 103303.


\bibitem{CESX23} G. Cipolloni, L. Erd\H{o}s, D. Schr\"oder and Y. Xu, \emph{On the rightmost eigenvalue of non-Hermitian random matrices}, Ann. Probab. \textbf{51} (2023), 2192--2242.

\bibitem{CR22} N. Crawford and R. Rosenthal, \emph{Eigenvector correlations in the complex Ginibre ensemble}, Ann. Appl. Probab. \textbf{32} (2022), 2706--2754.

\bibitem{Eier93} M. Eiermann, \emph{Fields of values and iterative methods}, Linear Algebra Appl. \textbf{180} (1993), 167--197.

\bibitem{Fo10} P. J. Forrester, \emph{Log-gases and random matrices}, Princeton University Press, Princeton, NJ, 2010.

\bibitem{FHS22} M. Fukuda, T. Hasebe and S. Sato, \emph{Additivity violation of quantum channels via strong convergence to semi-circular and circular elements}, Random Matrices Theory Appl. \textbf{11} (2022), 2250012. 

\bibitem{Fyo18} Y. V. Fyodorov,  \emph{On statistics of bi-orthogonal eigenvectors in real and complex Ginibre ensembles: combining partial Schur decomposition with supersymmetry}, Comm. Math. Phys. {\bf 363} (2018), 579--603.

\bibitem{FT21} Y. V. Fyodorov and W. Tarnowski, \emph{Condition numbers for real eigenvalues in the real elliptic Gaussian ensemble}, Ann. Henri Poincar\'{e} {\bf 22}, (2021), 309--330.

\bibitem{HT05} U. Haagerup and S. Thorbjørnsen, \emph{A new application of random matrices: $Ext(C^{*}_{red}(F_2))$ is not a group}, Ann. of Math. \textbf{162} (2005), 711--775.

\bibitem{GR97} K. E. Gustafson and D. K. M. Rao, \emph{Numerical Range: The Field of Values of Linear Operators and Matrices}, Springer-Verlag, NewYork, 1997.

\bibitem{HP00} F. Hiai and D. Petz, \emph{The semicircle law, free random variables and entropy}, Mathematical Surveys and Monographs, vol. 77, American Mathematical Society, Providence, RI, 2000.

\bibitem{HJ94} R. A. Horn and C. R. Johnson, \emph{Topics in Matrix Analysis}, Cambridge University Press, Cambridge, 1994. 

\bibitem{HM25} X. Hu and Y. Ma, \emph{Convergence rate of extreme eigenvalue of Ginibre ensembles to Gumbel distribution}, arXiv:2506.04560. 

\bibitem{John76} C. R. Johnson, \emph{Normality and the numerical range}, Linear Algebra Appl. \textbf{15} (1976), 89--94.

\bibitem{KS10} E. Kanzieper and N. Singh, \emph{Non-Hermitean Wishart random matrices (I)}, J. Math. Phys. \textbf{51} (2010), 103510. 

\bibitem{LW16} D.-Z. Liu and Y. Wang, \emph{Universality for products of random matrices I: Ginibre and truncated unitary cases}, Int. Math. Res. Not. IMRN \textbf{2016} (2016), 3473--3524.

\bibitem{MS17} J. A. Mingo and R. Speicher, \emph{Free probability and random matrices}, volume 35. Springer, 2017.

\bibitem{MKS24} S. Morimoto, M. Katori and T. Shirai, \emph{Generalized eigenspaces and pseudospectra of nonnormal and defective matrix-valued dynamical systems}, arXiv:2411.06472. 

\bibitem{NO15} H. H. Nguyen and S. O’Rourke, \emph{The elliptic law}, Int. Math. Res. Not. IMRN \textbf{2015} (2015), 7620--7689.

\bibitem{NS06} A. Nica and R. Speicher, \emph{Lectures on the combinatorics of free probability}, volume 335 of London Mathematical Society Lecture Note Series. Cambridge University Press, Cambridge, 2006.

\bibitem{Tad25} E. Tadmor, \emph{On the stability of Runge–Kutta methods for arbitrarily large systems of ODEs}, Comm. Pure Appl. Math. \textbf{78} (2025), 821--855.



\bibitem{ORSV14} S. O'Rourke, D. Renfrew, A.~Sohnikov and V.~Vu, \emph{Products of independent elliptic random matrices}, J. Stat. Phys. \textbf{160} (2015), 89--119.


\bibitem{Os04} J. C. Osborn, \emph{Universal results from an alternate random matrix model for QCD with a baryon chemical potential}, Phys. Rev. Lett. \textbf{93} (2004), 222001. 



\bibitem{Ste96} M. A. Stephanov, \emph{Random matrix model of QCD at finite density and the nature of the quenched limit}, Phys. Rev. Lett. \textbf{76} (1996), 4472.

\bibitem{XZ25} Y. Xu and Q. Zeng, \emph{Large deviations for the extremal eigenvalues of Ginibre ensembles}, arXiv:2512.12711. 

\end{thebibliography}

\end{document}